\documentclass[11.5pt]{amsart}
\usepackage{epsfig,graphics,graphicx,color,float,url}
\graphicspath{{images/}}
\usepackage{amsmath,amsthm,amsopn,amstext,amscd,amsfonts,amssymb}
\usepackage{comment}
\usepackage{subfig}
\usepackage[numbers,sort&compress]{natbib}
\bibpunct{[}{]}{,}{n}{}{;}
\usepackage{caption}
\usepackage{wrapfig}
\DeclareGraphicsExtensions{.eps}
\usepackage[export]{adjustbox}
\theoremstyle{plain}
\usepackage[english,activeacute]{babel}
\usepackage[latin1]{inputenc}
\usepackage{array}
\usepackage{pdflscape}
\usepackage{a4wide}
\usepackage{epstopdf}
\newcounter{myremark}

\numberwithin{myremark}{section}

\theoremstyle{plain}
\newtheorem{theorem}{Theorem}

\newtheorem{lemma}[theorem]{Lemma}
\newtheorem{proposition}[theorem]{Proposition}

\theoremstyle{definition}

\newtheorem{example}[theorem]{Example}

\newcommand{\NN}{\mathbb{N}}
\newcommand{\PP}{\mathbb{P}}
\newcommand{\RR}{\mathbb{R}}
\newcommand{\ZZ}{\mathbb{Z}}

\renewcommand{\span}[1]{\mbox{{\rm span\/}}\left\{#1\right\}}
\begin{document}

\title[Some relations between the Riemann zeta function and the generalized Bernoulli polynomials of level $m$]{Some relations between the Riemann zeta function and the generalized Bernoulli polynomials of level $m$}

\author[Y. Quintana]{Yamilet Quintana$^{(1)}$}
\address{\noindent Departamento de Matem\'aticas Puras y Aplicadas,
Edificio Matem\'aticas y Sistemas (MYS), Apartado Postal: 89000, Caracas 1080 A, Universidad Sim\'on Bol\'{\i}var,
Venezuela} \email{yquintana@usb.ve}
\thanks{$(1)\,\,\,$ Supported in part by Decanato de Investigaci\'on y Desarrollo, Universidad Sim\'on Bol\'{\i}var, Venezuela, grant DID-USB (S1-IC-CB-004-17).}

\author{H\'ector Torres-Guzm\'an}
\address{\noindent Departamento de Matem\'aticas Puras y Aplicadas,
Edificio Matem\'aticas y Sistemas (MYS), Apartado Postal: 89000, Caracas 1080 A, Universidad Sim\'on Bol\'{\i}var,
Venezuela}
\email{12-11307@usb.ve}

%\date{\textcolor{red}{{\bf Internal Draft No. 6. December 31th, %2018.}}}

\begin{abstract}
The main purpose of this paper is to show some relations between the Riemann zeta function and the generalized Bernoulli polynomials of level $m$. Our approach is based on the use of Fourier expansions for the periodic generalized Bernoulli functions of level $m$, as well as  quadrature formulae of Euler-Maclaurin type. Some illustrative examples involving such relations are also given.
\end{abstract}

\maketitle{}

\markboth{Y. QUINTANA AND H. TORRES-GUZM\'AN \mbox{}}{
Riemann zeta function and the generalized Bernoulli polynomials of level $m$}

\section{Introduction}

\label{[Section-1]-Introduction}

Let $\zeta(s)$ be the Riemann zeta function defined by
$$\zeta(s)=\sum_{n=1}^{\infty} \frac{1}{n^{s}}, \quad \Re(s)>1.$$
It is a classical result due to Riemann that $\zeta(s)$ can be analytically continued to a meromorphic function on the
whole complex plane with the only pole at $s = 1$, which is a simple pole with residue $1$. Also, if we consider the classical Bernoulli polynomials given by
$$\frac{z e^{xz}}{e^{z}-1} =\displaystyle\sum\limits_{n=0}^{\infty}
B_{n}(x)\frac{z^n}{n!}, \quad |z|<2\pi,$$
and the classical Bernoulli numbers, $B_{n}=B_{n}(0)$,  for all $n\geq 0$, then it is well known the following relation  between $\zeta(s)$ and the Bernoulli polynomials:

\begin{equation}
\label{equa1}
\zeta(2k)= \frac{(-1)^{k-1}\pi^{2k}2^{2k-1}}{(2k)!}B_{2k}, \quad k\geq 1.
\end{equation}

Euler's relation \eqref{equa1} provides an elegant formula for the explicit evaluation of $\zeta(2k)$, which shows the arithmetical nature of $\zeta(2k)$. However, for the odd zeta values $\zeta(2k+1)$ there is very little known information. For instance, in his paper of 1981 R. Ap\'ery showed that $\zeta(3)$ is irrational, but for $k\geq 2$ the arithmetical nature of $\zeta(2k+1)$ remains open (cf.  \cite{A1973,Ara,Ay,B2017,BS2017,CNRV2015} and the references thereof).

\medskip

In this contribution we are interested in exploring  similar relations to \eqref{equa1} in the setting of the generalized Bernoulli polynomials of level $m$ \cite{NB,QRU}. In order to do that, we show some constraints of the use of Fourier expansions for the periodic generalized Bernoulli functions of level $m$, as well as, our approach which is mainly based on  quadrature formulae of Euler-Maclaurin type.

\medskip

The outline of the paper as follows. Section \ref{[Section-2]-Stype} provides a short
background about some relevant properties of the generalized Bernoulli polynomials of level $m$.  Section \ref{[Section-3]-Basic} is devoted to show some constraints of the use of Fourier expansions for the periodic generalized Bernoulli functions of level $m$ (see Theorems  \ref{pzeta1} and \ref{htyq1}). Finally, Section  \ref{[Section-4]-Basic} contains the basic ideas in order to obtain quadrature formulae of Euler-Maclaurin type based on generalized Bernoulli polynomials of level $m$ (see Theorem \ref{htyq2}). Also, in this section is proved a result that reveals an interesting property about the applications of the quadrature formulae of Euler-Maclaurin type based on these polynomials (see Theorem \ref{teoconv}). As usual, throughout this paper the convention $0^{0}=1$ will be adopted and an empty sum will be interpreted to be  zero.

\section{Generalized Bernoulli polynomials of level $m$: some properties}
\label{[Section-2]-Stype}

For a fixed $m\in\NN$, the generalized Bernoulli polynomials of level $m$ are defined by means of the following generating function \cite{NB,QRU,SM,SC2012,Sri}
\begin{equation}
\label{mpol}
\displaystyle \frac{z^{m}e^{xz}}{e^{z}-\sum_{l=0}^{m-1}\frac{z^{l}}{l!}} =\displaystyle\sum\limits_{n=0}^{\infty}
B_n^{[m-1]}(x)\frac{z^n}{n!}, \quad |z|<2\pi
\end{equation}
and, the generalized Bernoulli numbers of level $m$ are defined by $B_n^{[m-1]}:= B_n^{[m-1]}(0)$, for all $n\geq 0$. It is clear that if $m=1$ in \eqref{mpol}, then  we obtain the definition of the classical Bernoulli polynomials $B_{n}(x)$, and classical Bernoulli numbers, respectively, i.e., $B_{n}(x)=B_{n}^{[0]}(x)$, and $B_{n}=B_{n}^{[0]}$, respectively,  for all $n\geq 0$.

\medskip

It is not difficult to check that the first four generalized Bernoulli polynomials of level $m$ are:

\begin{eqnarray*}
B_{0}^{[m-1]}(x)&=& m!,\\
B_{1}^{[m-1]}(x)&=& m!\left(x-\frac{1}{m+1}\right),\\
B_{2}^{[m-1]}(x)&=& m!\left( x^{2}- \frac{2}{m+1}x+\frac{2}{(m+1)^{2}(m+2)}\right),\\
B_{3}^{[m-1]}(x)&=& m!\left( x^{3}- \frac{3}{m+1}x^{2} +\frac{6}{(m+1)^{2}(m+2)}x+ \frac{6(m-1)}{(m+1)^{3}(m+2)(m+3)}\right).
\end{eqnarray*}

\medskip

The following results summarizes some properties of the generalized Bernoulli polynomials of level $m$ (cf. \cite{NB,HQU,QRU}).

\begin{proposition}
\label{elementaryp}
\cite[Proposition 1]{QRU} For a fixed $m\in\NN$, let $\left\{B_n^{[m-1]}(x)\right\}_{n\geq 0}$  be the sequence of generalized  Bernoulli polynomials of level $m$. Then the following statements hold:

\begin{enumerate}
\item[a)] Summation formula. For every $n\geq 0$,
\begin{eqnarray}
\label{[Sec2]-bsf1}
B_n^{[m-1]}(x)&=&\sum_{k=0}^{n}\binom{n}{k}B_{k}^{[m-1]}\,x^{n-k}.
\end{eqnarray}

\item[b)] Differential relations (Appell polynomial sequences). For  $n,j\geq 0$ with $0\leq j\leq n$, we have
\begin{equation}
\label{[Sec2]-AppDer1}
\lbrack B_n^{[m-1]}(x)]^{(j)}=\frac{n!}{(n-j)!}\, B_{n-j}^{[m-1]}(x).
\end{equation}

\item[c)] Inversion formula. \cite[Equation (2.6)]{NB} For every $n\geq 0$,
\begin{equation}
\label{[Sec2]-recu1}
x^{n}= \sum_{k=0}^{n}\binom{n}{k}\frac{k!}{(m+k)!}B_{n-k}^{[m-1]}(x).
\end{equation}

\item[d)] Recurrence relation. \cite[Lemma 3.2]{NB} For every $n\geq 1$,
\begin{equation*}
\label{[Sec2]-recu2}
B_n^{[m-1]}(x)= \left( x- \frac{1}{m+1}\right)B_{n-1}^{[m-1]}(x)-\frac{1}{n(m-1)!} \sum_{k=0}^{n-2}\binom{n}{k}B_{n-k}^{[m-1]}B_{k}^{[m-1]}(x).
\end{equation*}

\item[e)] Integral formulas.
\begin{eqnarray}
\label{[Sec2]-Intf1}
\int_{x_{0}}^{x_{1}} B_n^{[m-1]}(x)dx&=&\frac{1}{n+1}\left[B_{n+1}^{[m-1]}(x_{1})-B_{n+1}^{[m-1]}(x_{0})\right]
\end{eqnarray}
\begin{equation*}
\label{[Sec2]-Intf2}
=\sum_{k=0}^{n}\frac{1}{n-k+1}\binom{n}{k}
B_{k}^{[m-1]}((x_{1})^{n-k+1}-(x_{0})^{n-k+1}).
\end{equation*}

\begin{eqnarray}
\label{[Sec2]-Intf3}
B_{n}^{[m-1]}(x)&=&n\int_{0}^{x} B_{n-1}^{[m-1]}(t) dt + B_{n}^{[m-1]}.
\end{eqnarray}
\item[f)] \cite[Theorem 3.1]{NB} Differential equation. For every $n\geq 1$, the polynomial $B_{n}^{[m-1]}(x)$ satisfies the following differential equation
\begin{equation*}
\label{[Sec2]-diffe1}
\frac{B_{n}^{[m-1]}}{n!}y^{(n)}+ \frac{B_{n-1}^{[m-1]}}{(n-1)!}y^{(n-1)}+\cdots+ \frac{B_{2}^{[m-1]}}{2!}y''+ (m-1)!\left( \frac{1}{m+1}-x\right)y'+ n(m-1)!y=0.
\end{equation*}
\end{enumerate}
\end{proposition}

If we denote by $\PP_n$ the linear space of polynomials with real coefficients and degree less than or equal to $n$, then  \eqref{[Sec2]-recu1} implies that

\begin{proposition}
\label{bas1} \cite[Proposition 2]{QRU} For a fixed $m\in\NN$ and each $n\geq 0$, the set\\   $\left\{B_{0}^{[m-1]}(x), B_{1}^{[m-1]}(x), \ldots, B_{n}^{[m-1]}(x)\right\}$ is a basis for $\PP_{n}$, i.e.,
$$\PP_{n}=\mathfrak{B}^{[m-1]}_{n}= \span{B_{0}^{[m-1]}(x), B_{1}^{[m-1]}(x), \ldots, B_{n}^{[m-1]}(x)}.$$
\end{proposition}

We conclude this section showing in Figure \ref{f:figur1}  the plots of some generalized  Bernoulli polynomials of level $m$.

\begin{figure}[H]
  \centering
  \subfloat[Level: $m=1$. Degrees: $n=3$ (black), $n=4$ (green), $n=5$ (blue), $n=10$ (red).]{
  \label{f:primera}
  \includegraphics[width=0.4\textwidth]{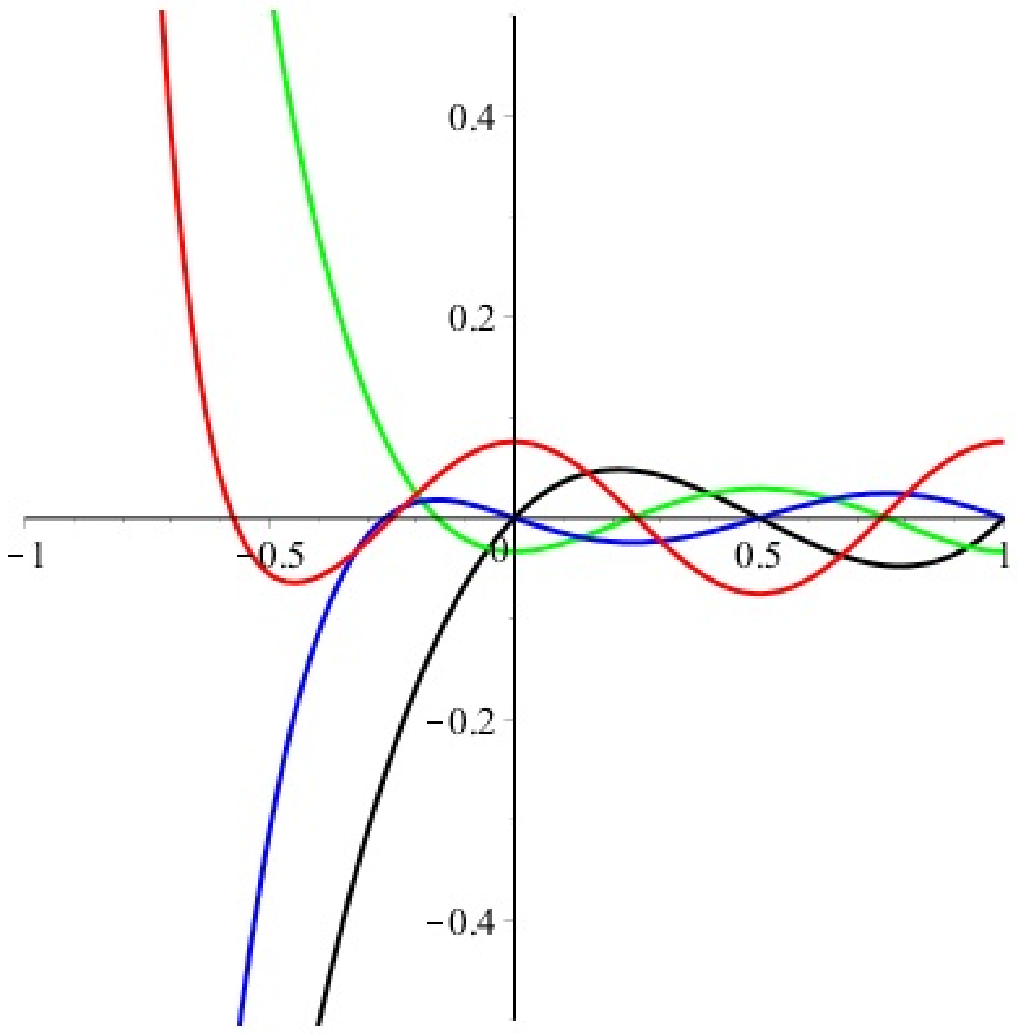}}
  \hspace{8mm}
\subfloat[Level: $m=5$. Degrees: $n=2$ (green), $n=5$ (blue), $n=6$ (black), $n=10$ (red).]{
  \label{f:segunda}
  \includegraphics[width=0.4\textwidth]{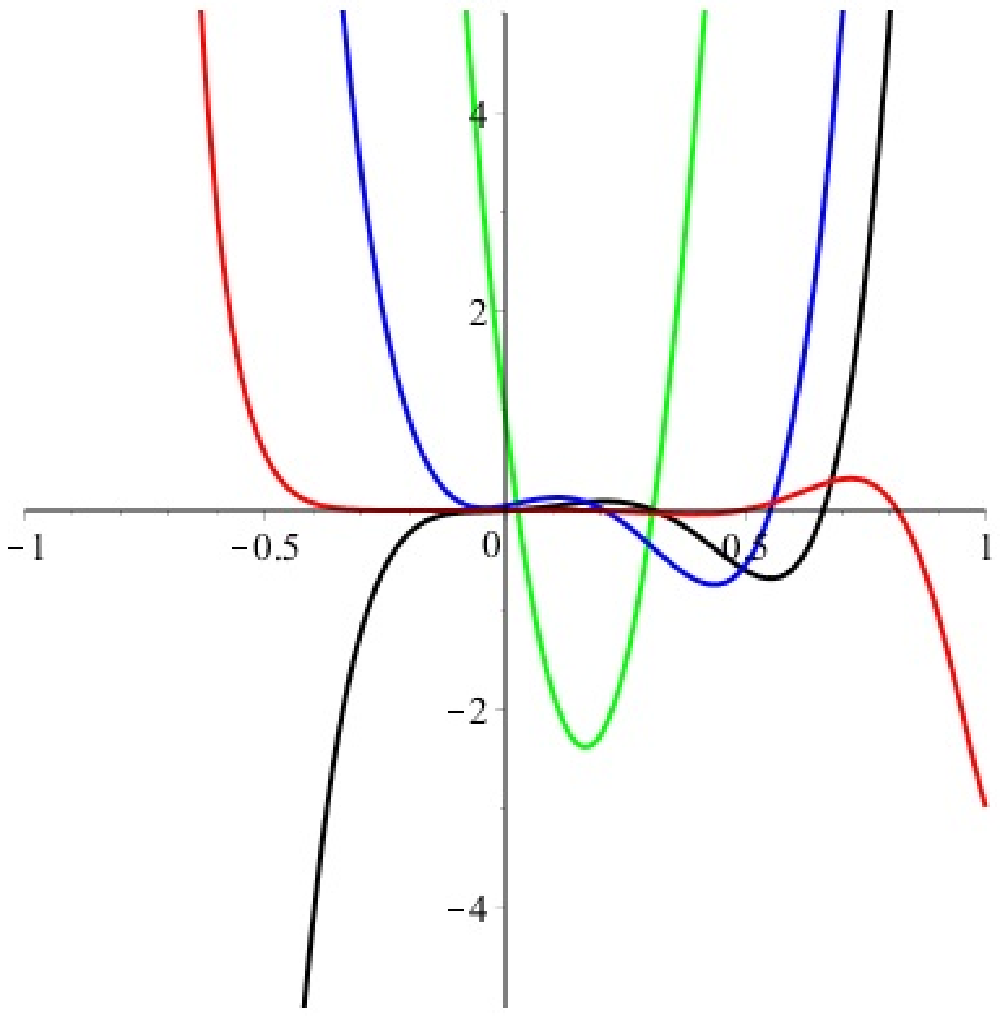}}
\vspace{-3mm}
\caption{Graphs of some generalized  Bernoulli polynomials for the levels $m=1$ (classical Bernoulli polynomials) and $m=5$, respectively.}
\label{f:figur1}
\end{figure}

\section{Fourier expansions and generalized Bernoulli polynomials of level $m$}
\label{[Section-3]-Basic}

For a fixed $m\in\NN$, the periodic generalized Bernoulli functions of level $m$ are defined as follows.
\begin{equation}
\label{peri1}
p_{n}^{[m-1]}(x)=\frac{B_{n}^{[m-1]}(x)}{n!}, \quad 0\leq x<1, \qquad p_{n}^{[m-1]}(x+1)=p_{n}^{[m-1]}(x), \quad x\in \RR.
\end{equation}

The functions $p_{n}^{[m-1]}(x)$ are continous on $\RR$ with continous derivatives up to order $n-1$ only if $m=1$ and $n\geq 2$.

\medskip

In what follows, the symbol ``$\sim$'' is used to refer to the formal Fourier expansion for a given function on an interval, and it is not associated to some notion of convergence in particular, since as we know there are several kinds of convergence involved with the notion of Fourier expansion associated to a given function.

\medskip

For $m=1$ the Fourier expansions for the periodic generalized Bernoulli functions of level $m$ coincide with the Fourier expansions for the periodic Bernoulli functions, i.e.,
\begin{eqnarray}
\label{peri2}
 p_{1}^{[0]}(x) = p_{1}(x)  &\sim& - \sum_{k=1}^{\infty}\frac{2\sin(2\pi k x)}{2\pi k},\\
\label{peri3}
 p_{2r}^{[0]}(x) = p_{2r}(x)&=& (-1)^{r-1}\sum_{k=1}^{\infty}\frac{2\cos(2\pi k x)}{(2\pi k)^{2r}}, \quad r\geq 1, \\
\label{peri4}
p_{2r+1}^{[0]}(x) = p_{2r+1}(x) &=& (-1)^{r-1}\sum_{k=1}^{\infty}\frac{2\sin(2\pi k x)}{(2\pi k)^{2r+1}}, \quad r\geq 1.
\end{eqnarray}

Notice that by a well known result on the uniform convergence of
Fourier expansions (see, for instance, \cite{Foll,Ph,St}), the  Fourier series \eqref{peri3} and \eqref{peri4} are uniformly
convergent, while  this does not hold for the Fourier expansion \eqref{peri2}, since
$$p_{1}(0)=p_{1}(1)=-\frac{1}{2}, \quad \mbox{ and }  \quad \lim_{\epsilon\to 0^{+}}p_{1}(1-\epsilon)=\frac{1}{2},$$
whereas the Fourier expansion \eqref{peri2} assumes the value $0$ at both $x=0$ and $x=1$. Figure \ref{f:figur2}  shows the plots for some periodic Bernoulli functions.

\begin{figure}[H]
  \centering
  \subfloat[Graph of $p_{1}(x)$.]{
  \label{f:segunda}
  \includegraphics[width=0.4\textwidth]{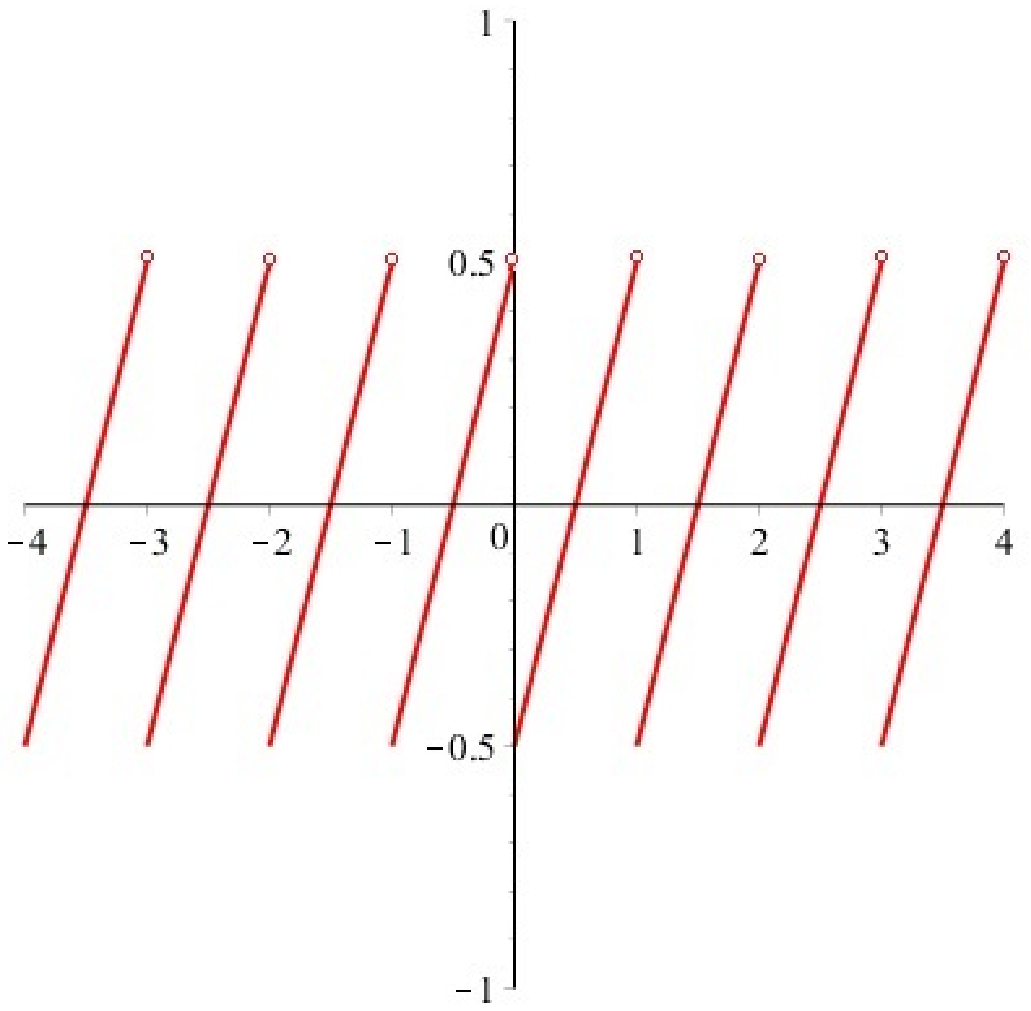}}
  \hspace{10mm}
  \subfloat[Graph of $p_{5}(x)$.]{
  \label{f:tercera}
  \includegraphics[width=0.4\textwidth]{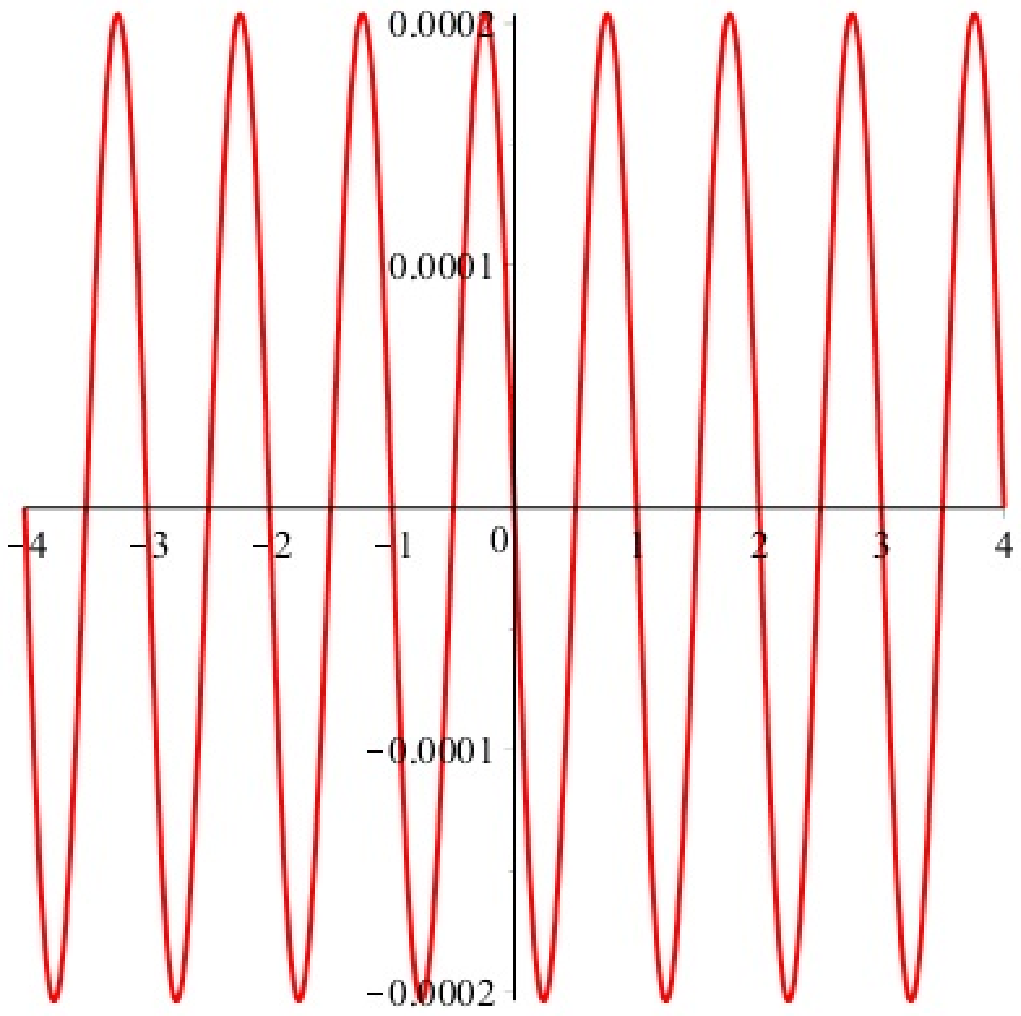}}
 \vspace{5mm}
  \subfloat[Graph of $p_{10}(x)$.]{
  \label{f:cuarta}
  \includegraphics[width=0.4\textwidth]{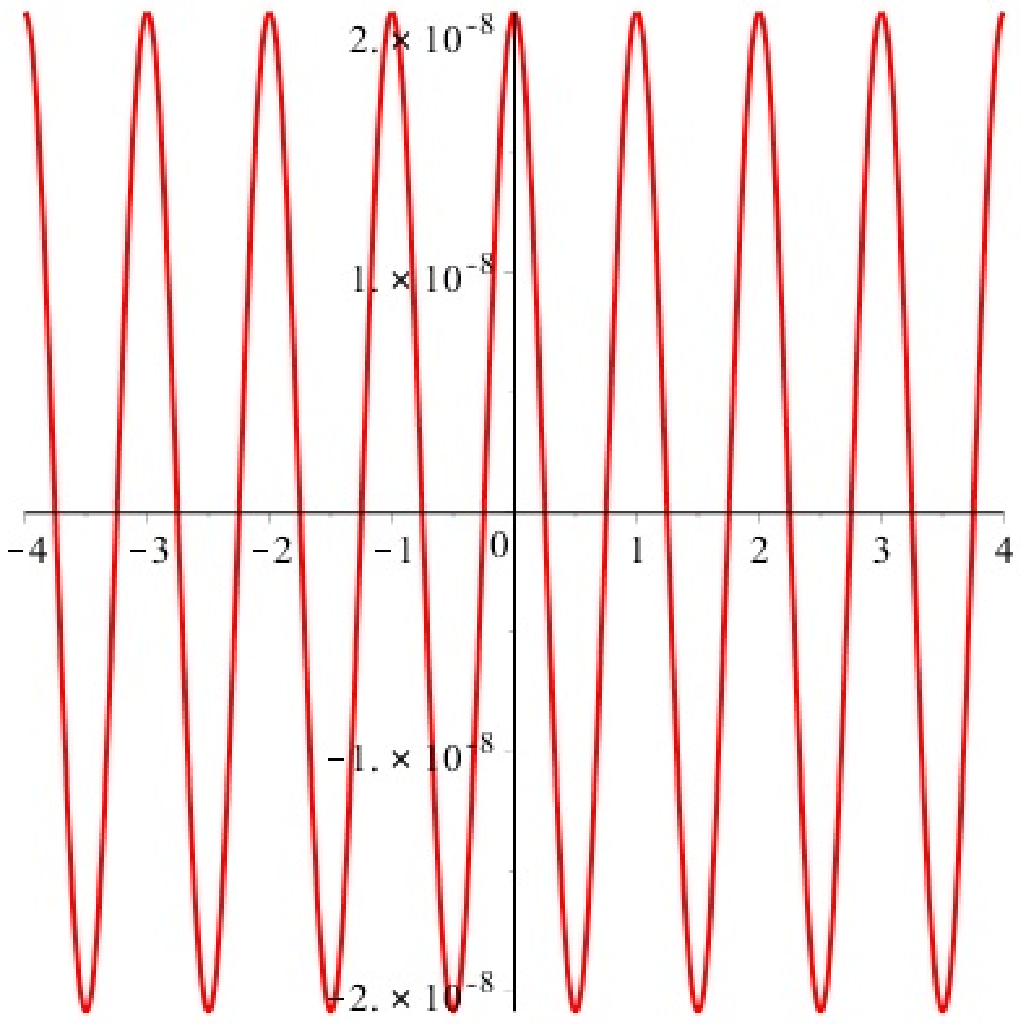}}
\vspace{-2mm}
 \caption{Periodic Bernoulli functions for $n=1,5,10$.}
\label{f:figur2}
\end{figure}

It is important to note that the sequence of functions $\left\{p_{n}(x)\right\}_{n\geq 2}\subset C^{n-2}(-\infty,\infty)$ (when $n=2$, we are using the notation $C^{0}(-\infty,\infty)=C(-\infty,\infty)$), because the  Bernoulli numbers satisfy the equality $B_{n}=(-1)^{n}B_{n}(1)$, for any $n\geq 0$ (see e.g., \cite[Proposition 4.9]{Ara}), $B_{n}=0$, if $n\geq 3$ is odd, and by the condition of periodicity  \eqref{peri1} with $m=1$. In Figure \ref{f:figur3}  the plots for several generalized  Bernoulli polynomials of level $m=5$ and their corresponding periodic generalized Bernoulli functions are shown.

\begin{figure}[H]
  \centering
  \subfloat[Graph of $B_{1}^{[4]}(x)$.]{
  \label{f:quinta}
  \includegraphics[width=0.34\textwidth]{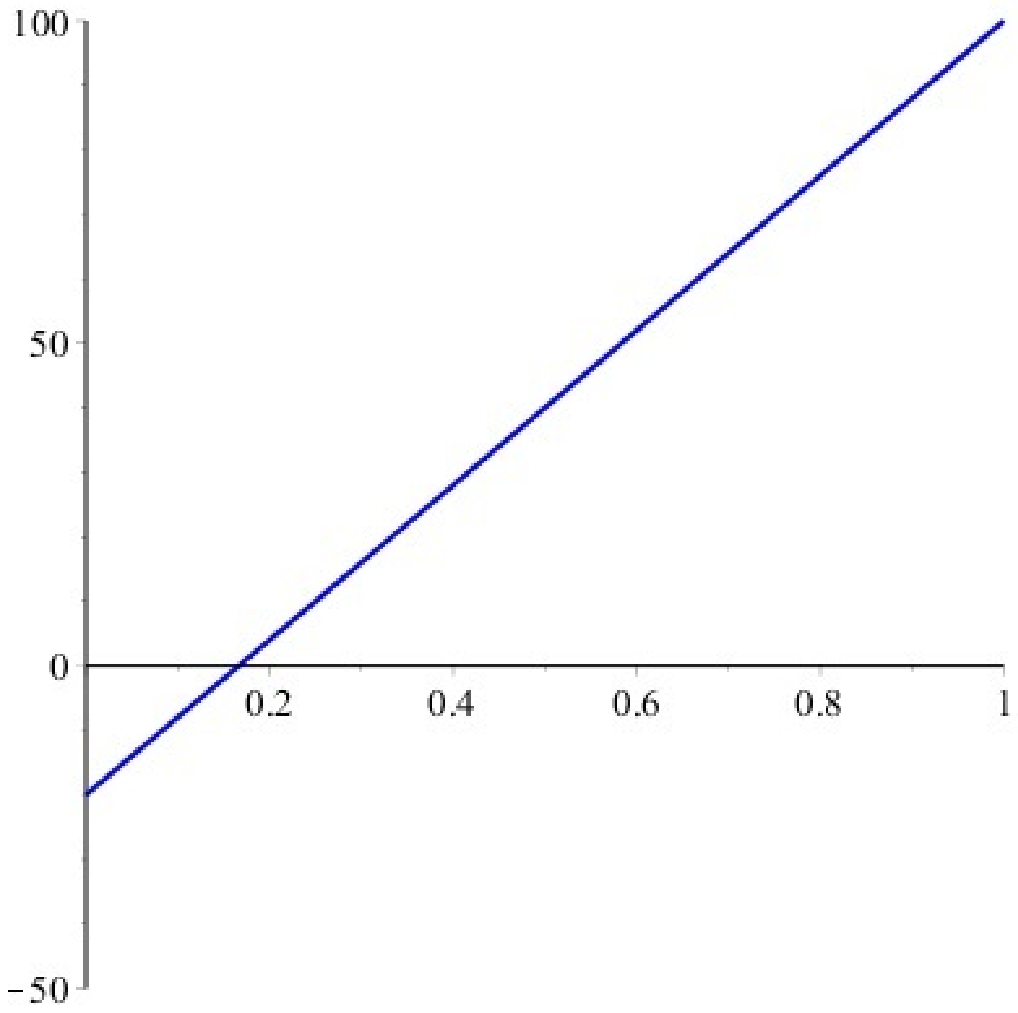}}
  \hspace{10mm}
  \subfloat[Graph of $p_{1}^{[4]}(x)$.]{
  \label{f:sexta}
  \includegraphics[width=0.34\textwidth]{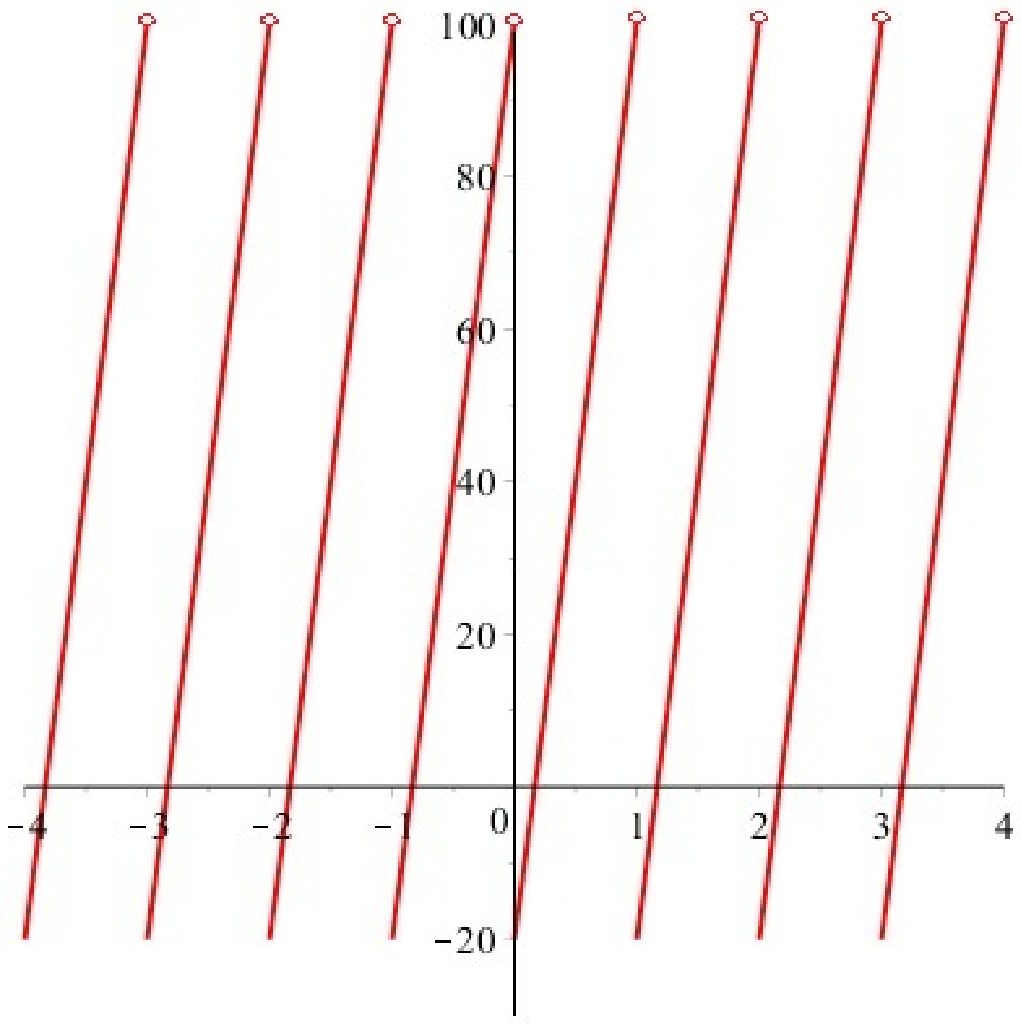}}
 \vspace{5mm}
  \subfloat[Graph of $B_{4}^{[4]}(x)$.]{
  \label{f:septima}
  \includegraphics[width=0.34\textwidth]{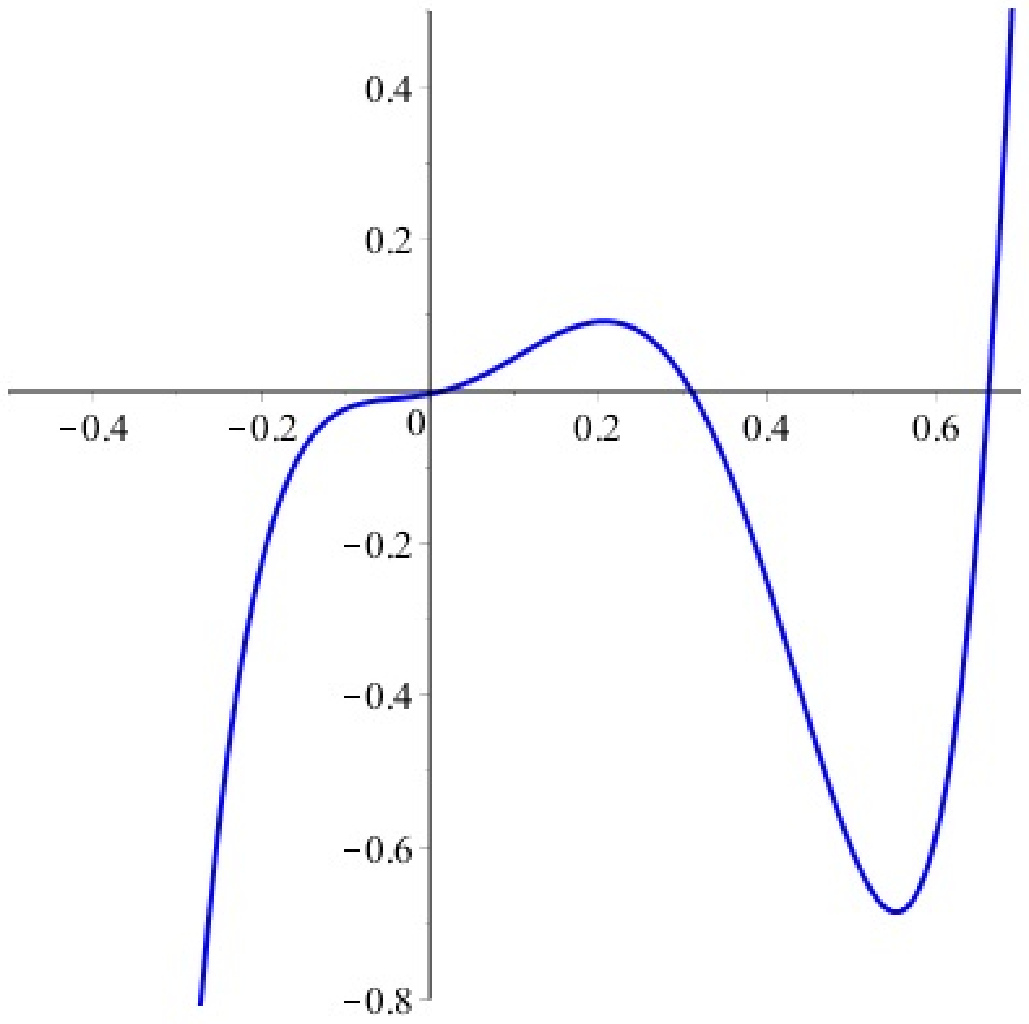}}
\hspace{10mm}
  \subfloat[Graph of $p_{4}^{[4]}(x)$.]{
  \label{f:octava}
  \includegraphics[width=0.34\textwidth]{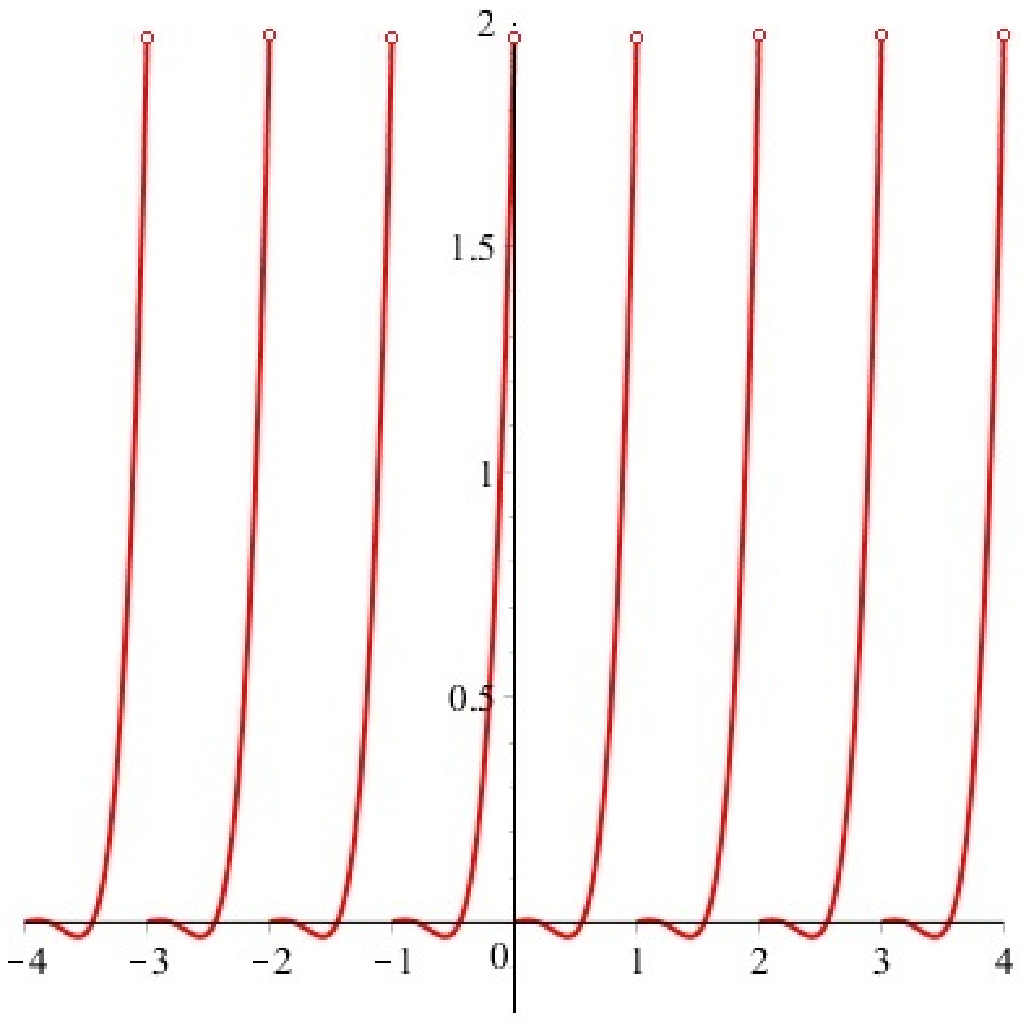}}
 \vspace{5mm}
 \subfloat[Graph of $B_{10}^{[4]}(x)$.]{
  \label{f:novena}
  \includegraphics[width=0.34\textwidth]{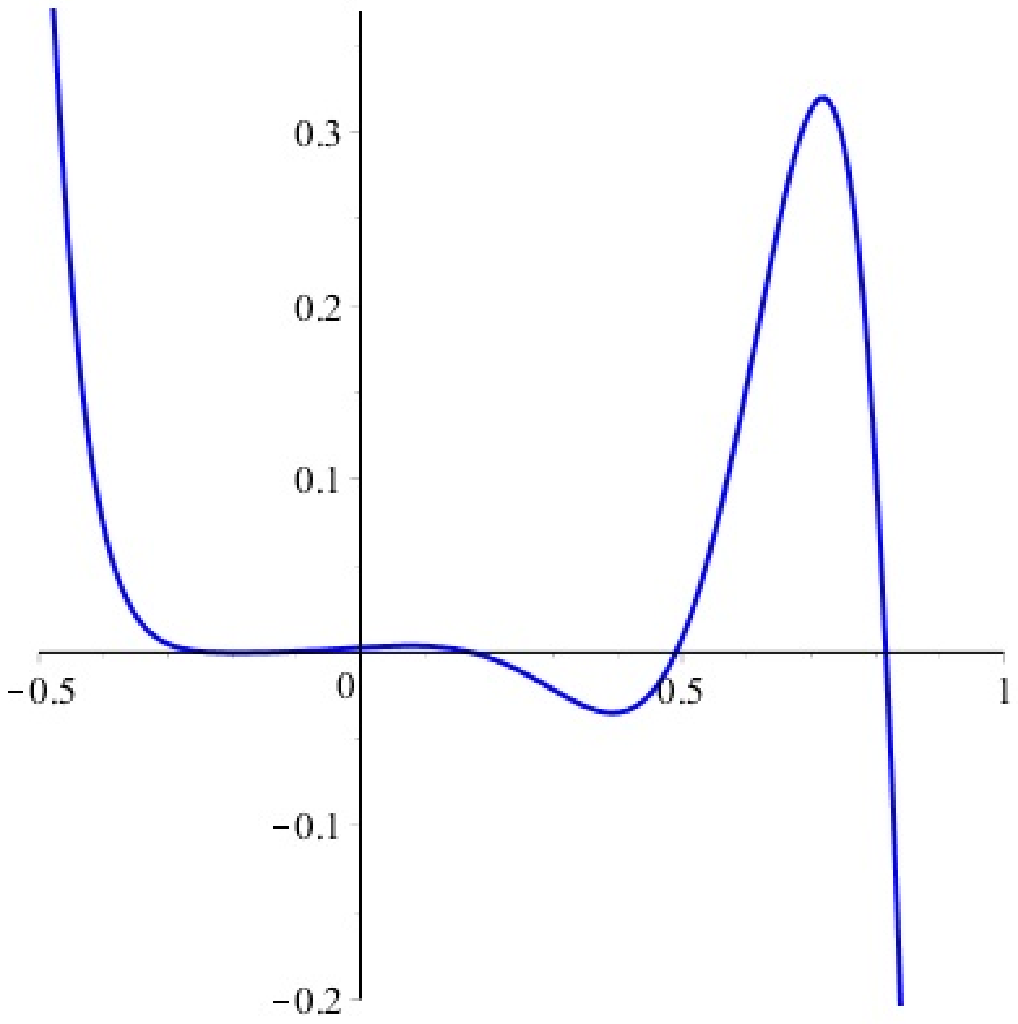}}
\hspace{10mm}
  \subfloat[Graph of $p_{10}^{[4]}(x)$.]{
  \label{f:decima}
  \includegraphics[width=0.34\textwidth]{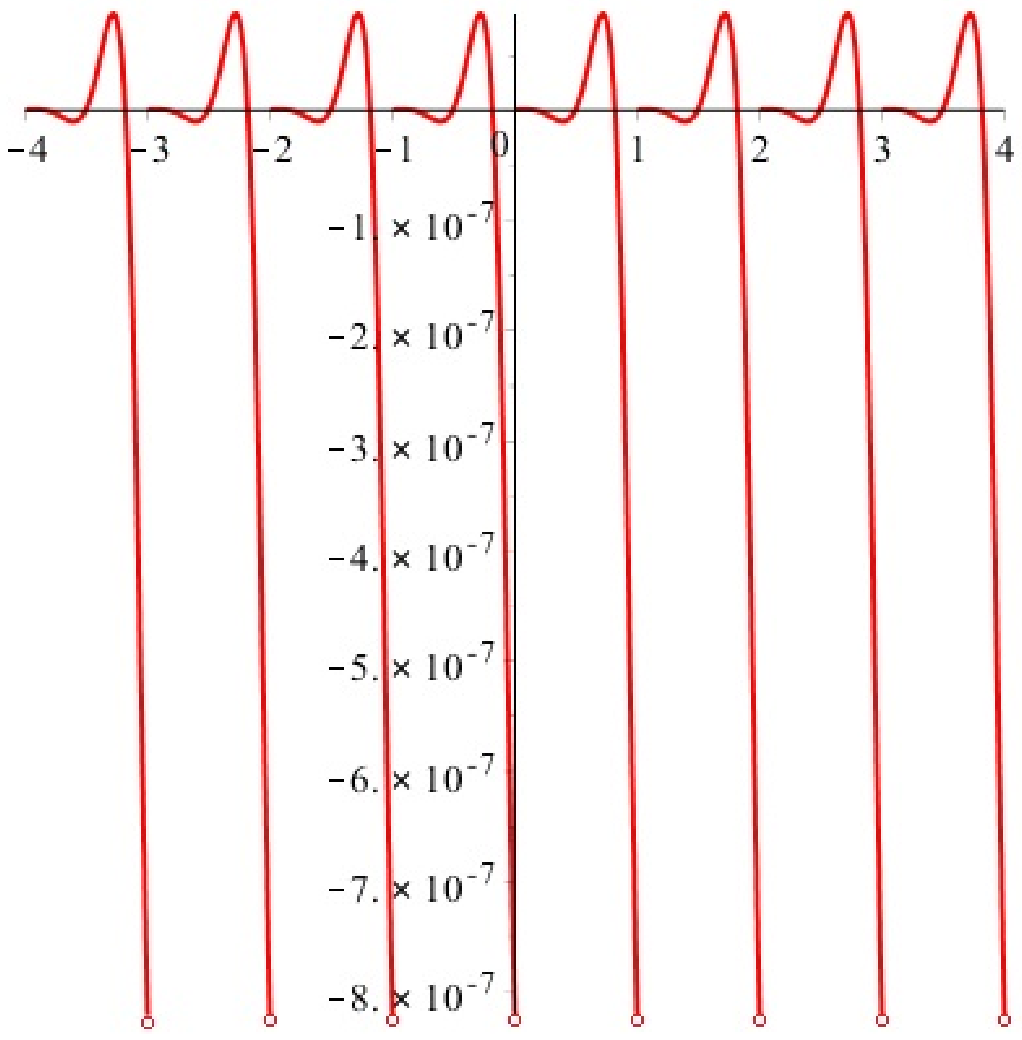}}
\vspace{-2mm}
 \caption{Generalized  Bernoulli polynomials of level $m=5$ and their corresponding periodic generalized Bernoulli functions for $n=1,4,10$.}
\label{f:figur3}
\end{figure}
It is worty to mention that for $m,n>1$ the functions  $p_{n}^{[m-1]}(x)$ are only differentiable on  $\RR\setminus\ZZ$ -unlike what happens when  $m=1$ and $n>1$ are considered (cf. \cite[Chap. 3, Sec. 3.2]{Ph})-. Thus, from \eqref{peri1} and \eqref{[Sec2]-AppDer1} we deduce that $\left[p_{n+1}^{[m-1]}(x)\right]^{'}= p_{n}^{[m-1]}(x)$ for each $x\in (k,k+1)$, $k\in\ZZ$. Hence,
\begin{equation}
\label{DervFPB}
\left[p_{n+1}^{[m-1]}(x)\right]^{'}= p_{n}^{[m-1]}(x),\quad\mbox{ if } \quad x\in \RR\setminus\ZZ.
\end{equation}

Also, the periodic generalized Bernoulli functions of level $m$
are  integrable function on  $[0,1]$. Therefore, they satisfy  Dirichlet conditions for the existence of their Fourier expansions
\cite{Foll,St}.

\medskip

For a fixed $m\in\NN$ we note that $p_{1}^{[m-1]}(x)$ has the following Fourier coefficients:
\begin{eqnarray*}
a_{0,1}^{[m-1]}&=& 2\int_{0}^{1}p_{1}^{[m-1]}(x)\,dx= \frac{m!}{2}\left(\frac{m-1}{m+1}\right), \\
a_{k,1}^{[m-1]}&=& 2\int_{0}^{1}p_{1}^{[m-1]}(x)\cos(2\pi k x)\,dx =0, \quad k\geq 1,\\
b_{k,1}^{[m-1]}&=& 2\int_{0}^{1}p_{1}^{[m-1]}(x)\sin(2\pi k x)\,dx=-\frac{2m!}{2\pi k}, \quad k\geq 1.
\end{eqnarray*}
Thus, $p_{1}^{[m-1]}(x)$ has the Fourier expansion
\begin{equation}
\label{peri6}
p_{1}^{[m-1]}(x)\sim \frac{m!}{2}\left(\frac{m-1}{m+1}\right) - \sum_{k=1}^{\infty}\frac{2m!\sin (2\pi k x)}{2\pi k}.
\end{equation}

For $x\in (0,1)$, let us integrate the series \eqref{peri6} formally, term by term:
\begin{eqnarray}
\nonumber
\int_{0}^{x} p_{1}^{[m-1]}(t)\,dt&=& \frac{m!}{2}\left(\frac{m-1}{m+1}\right)x- \sum_{k=1}^{\infty}\frac{2m!}{2\pi k}\int_{0}^{x}\sin(2\pi kt)\,dt \\
\nonumber
 &=& \frac{m!}{2}\left(\frac{m-1}{m+1}\right)x-\sum_{k=1}^{\infty}
\frac{2m!}{(2\pi k)^{2}}(1-\cos(2\pi kx))\\
 &=&
\label{peri7}
\frac{m!}{2}\left(\frac{m-1}{m+1}\right)x -\frac{m!}{2\pi^{2}}\zeta(2)+ \sum_{k=1}^{\infty}
\frac{2m! \cos(2\pi kx)}{(2\pi k)^{2}}.
\end{eqnarray}
From \eqref{[Sec2]-Intf1} we have
\begin{equation}
\label{peri8}
\int_{0}^{x} p_{1}^{[m-1]}(t)\,dt=  p_{2}^{[m-1]}(x)- \frac{B_{2}^{[m-1]}}{2}.
\end{equation}

Hence, the substitution of \eqref{peri8} into \eqref{peri7} yields the following expansion for $p_{2}^{[m-1]}(x)$
\begin{equation}
\label{peri9}
p_{2}^{[m-1]}(x)= \frac{B_{2}^{[m-1]}}{2}+\frac{m!}{2}\left(\frac{m-1}{m+1}\right)x -\frac{m!}{2\pi^{2}}\zeta(2)+ \sum_{k=1}^{\infty}
\frac{2m! \cos(2\pi kx)}{(2\pi k)^{2}}.
\end{equation}

Since, $p_{2}^{[m-1]}(x)$ has the following Fourier coefficients:
\begin{eqnarray*}
a_{0,2}^{[m-1]}&=& 2\int_{0}^{1}p_{1}^{[m-1]}(x)\,dx= \frac{m!}{3}\left(\frac{m-1}{(m+1)^{2}}\right)\left(\frac{m^{2}+2m-2}{m+2}\right), \\
a_{k,2}^{[m-1]}&=& 2\int_{0}^{1}p_{2}^{[m-1]}(x)\cos(2\pi k x)\,dx =\frac{2m!}{(2\pi k)^{2}}, \quad k\geq 1,\\
b_{k,2}^{[m-1]}&=& 2\int_{0}^{1}p_{1}^{[m-1]}(x)\sin(2\pi k x)\,dx=-\frac{m!}{2\pi k}\left(\frac{m-1}{m+1}\right), \quad k\geq 1,
\end{eqnarray*}
then $p_{2}^{[m-1]}(x)$ has the Fourier expansion

\small{
\begin{equation}
\label{peri10}
p_{2}^{[m-1]}(x)= \frac{m!}{6}\left(\frac{m-1}{(m+1)^{2}}\right)\left(\frac{m^{2}+2m-2}{m+2}\right)
+\sum_{k=1}^{\infty}\frac{2m!\cos(2\pi kx)}{(2\pi k)^{2}}
-\sum_{k=1}^{\infty} \frac{m!(m-1)\sin(2\pi kx)}{2\pi k(m+1)}.
\end{equation}}

On comparing \eqref{peri9} and \eqref{peri10}, for $x\in (0,1)$ we see that

\begin{equation}
\label{peri11}
\frac{m!}{2\pi^{2}}\zeta(2)-\frac{B_{2}^{[m-1]}}{2}+\frac{m!}{6}\left(\frac{m-1}{(m+1)^{2}}\right)\left(\frac{m^{2}+2m-2}{m+2}\right)
=\frac{m!}{2}\left(\frac{m-1}{m+1}\right)x+\sum_{k=1}^{\infty} \frac{m!(m-1)\sin(2\pi kx)}{2\pi k(m+1)}.
\end{equation}

If we put $x=\frac{1}{2}$ in \eqref{peri11}, then we obtain
\begin{equation}
\label{peri12}
\zeta(2)=\frac{2\pi^{2}}{m!}\left[\frac{B_{2}^{[m-1]}}{2}
+\frac{m!}{4}\left(\frac{m-1}{m+1}\right)
-\frac{m!}{6}\left(\frac{m-1}{(m+1)^{2}}\right)\left(\frac{m^{2}+2m-2}{m+2}\right)
\right].
\end{equation}

The relation \eqref{peri12} connects the zeta number $\zeta(2)$ with the generalized Bernoulli polynomial $B_{2}^{[m-1]}(x)$ for any $m>1$. Notice that if $m=1$ then \eqref{peri12} coincides with Euler's relation \eqref{equa1} for $k=1$.

\medskip

For example, if we take $m=2$ then \eqref{peri12} becomes
\begin{equation*}
\zeta(2)=\pi^{2}\left( \frac{B_{2}^{[1]}}{2}+\frac{1}{6}-\frac{1}{18}\right)=\frac{\pi^{2}}{6}.
\end{equation*}

Since on $[0,1]$, the polynomial $p_{n}(x)$ is symmetric about the midpoint $x =\frac{1}{2}$, when $n$ is even, and it is antisymmetric about $x =\frac{1}{2}$, when $n$ is odd; that is,
\begin{equation}
\label{peri13}
p_{n}(1-x)=(-1)^{n}p_{n}(x), \quad 0\leq x\leq 1, \quad n\geq 2.
\end{equation}

It follows that when $m=1$,  taking $x=0$ in \eqref{peri3} and evaluating $p_{2r}(0)$ from \eqref{peri1} and using \eqref{peri13}, we obtain (cf. \cite[Eq. (3.54)]{Ph}):

\begin{equation*}
\label{peri14}
\zeta(2r)=\sum_{n=1}^{\infty}\frac{1}{n^{2r}}=(-1)^{r-1}\pi^{2r}2^{2r-1}
\frac{B_{2r}}{(2r)!}, \quad r\geq 1,
\end{equation*}
this last equation is precisely \eqref{equa1}.

\medskip

Next,  we will use the notation $p_{n}^{[m-1]}(x^{-})$ and $p_{n}^{[m-1]}(x^{+})$  for representing the one-sided limits
$\lim_{y\to x^{-}}p_{n}^{[m-1]}(y)$ and $\lim_{y\to x^{+}}p_{n}^{[m-1]}(y)$, respectively. The following Proposition provides the Fourier expansion for $p_{n}^{[m-1]}(x)$ when $m>1$.

\begin{proposition}
\label{FourierFPBG}
For a fixed $m\in\NN$ and any $n\in\NN$, let $p_{n}^{[m-1]}(x)$  be the periodic generalized Bernoulli functions of level $m$. Then Fourier expansion for $p_{n}^{[m-1]}(x)$ on $[0,1]$ is given by
\begin{equation}
\label{Fourier1}
p_{n}^{[m-1]}(x) \sim \frac{a_{0,n}^{[m-1]}}{2} + \sum_{k=1}^{\infty} a_{k,n}^{[m-1]}\cos(2\pi kx)   + b_{k,n}^{[m-1]} \sin(2\pi kx),
\end{equation}
where
\begin{equation}
\label{Fourier2}
\frac{a_{0,n}^{[m-1]}}{2} = p_{n+1}^{[m-1]}(1^{-})-p_{n+1}^{[m-1]}(0) = \frac{1}{(n+1)!}\left(B_{n+1}^{[m-1]}(1)- B_{n+1}^{[m-1]}\right).
\end{equation}
And for $k\geq 1:$
\begin{eqnarray}
\label{Fourier3}
a_{k,n}^{[m-1]} & = &\sum_{j=0}^{\lfloor \frac{n}{2} \rfloor -1} (-1)^{j}\frac{2}{(2\pi k)^{2j+2}}\left(p_{n-2j-1}^{[m-1]}(1^{-}) - p_{n-2j-1}^{[m-1]}(0)\right)\\
\label{Fouriers3}
 & = &\sum_{j=0}^{\lfloor \frac{n}{2} \rfloor -1} (-1)^{j}\frac{2}{(2\pi k)^{2j+2}}\frac{\left(B_{n-2j-1}^{[m-1]}(1)- B_{n-2j-1}^{[m-1]}\right)}{(n-2j-1)!},\\
\label{Fourier4}
b_{k,n}^{[m-1]} & = &\sum_{j=0}^{\lfloor \frac{n}{2} \rfloor} (-1)^{j+1}\frac{2}{(2\pi k)^{2j+1}}\left(p_{n-2j}^{[m-1]}(1^{-}) - p_{n-2j}^{[m-1]}(0)\right)\\
\label{Fouriers4}
& = &\sum_{j=0}^{\lfloor \frac{n}{2} \rfloor} (-1)^{j+1}\frac{2}{(2\pi k)^{2j+1}}\frac{\left(B_{n-2j}^{[m-1]}(1)- B_{n-2j}^{[m-1]}\right)}{(n-2j)!}.
\end{eqnarray}
\end{proposition}

\begin{proof}
For each $p_{n}^{[m-1]}(x)$ it is well known that its Fourier coefficients are  given by
\begin{eqnarray}
\label{aux1}
a_{0,n}^{[m-1]}&=& 2\int_{0}^{1}p_{n}^{[m-1]}(x)\,dx,\\
\label{aux2}
a_{k,n}^{[m-1]}&=& 2\int_{0}^{1}p_{n}^{[m-1]}(x)\cos(2\pi kx)\,dx, \quad k\geq 1,\\
\label{aux3}
b_{k,n}^{[m-1]}&=& 2\int_{0}^{1}p_{n}^{[m-1]}(x)\sin(2\pi kx)\,dx, \quad k\geq 1.
\end{eqnarray}

Then, \eqref{Fourier2} is a straightforward consequence of  \eqref{aux1} and \eqref{DervFPB}. For obtaining the relations \eqref{Fourier3} and \eqref{Fourier4} it suffices use integration by parts on the right-hand side of \eqref{aux2} and \eqref{aux3}, respectively. So, we get
\begin{eqnarray}
\label{aux4}
a_{k,n}^{[m-1]}&=&-\frac{1}{2\pi k}b_{k,n-1}^{[m-1]},\\
\label{aux5}
b_{k,n}^{[m-1]}&=&-\frac{2}{2\pi k}\left(p_{n}^{[m-1]}(1^{-}) - p_{n}^{[m-1]}(0)\right)+\frac{1}{2\pi k}a_{k,n-1}^{[m-1]}.
\end{eqnarray}

Then replacing $n$ by $n-1$ in \eqref{aux5} and substituting the result obtained into  \eqref{aux4}, we get the following recurrence relation
\begin{equation}
\label{aux6}
a_{k,n}^{[m-1]}+\frac{1}{(2\pi k)^{2}}a_{k,n-2}^{[m-1]}= \frac{2}{(2\pi k)^{2}}\left(p_{n-1}^{[m-1]}(1^{-}) - p_{n-1}^{[m-1]}(0)\right).
\end{equation}

Analogously, we can obtain
\begin{equation}
\label{aux7}
b_{k,n}^{[m-1]}+\frac{1}{(2\pi k)^{2}}b_{k,n-2}^{[m-1]}= -\frac{2}{2\pi k}\left(p_{n}^{[m-1]}(1^{-}) - p_{n}^{[m-1]}(0)\right).
\end{equation}

Finally, it follows  from \eqref{aux6} and \eqref{aux7} that
\begin{eqnarray*}
a_{k,n}^{[m-1]}&=& \frac{2}{(2\pi k)^{2}}\left(p_{n-1}^{[m-1]}(1^{-}) - p_{n-1}^{[m-1]}(0)\right)-\frac{2}{(2\pi k)^{4}}\left(p_{n-3}^{[m-1]}(1^{-}) - p_{n-3}^{[m-1]}(0)\right)\\
& &+ \frac{2}{(2\pi k)^{6}}\left(p_{n-5}^{[m-1]}(1^{-}) - p_{n-5}^{[m-1]}(0)\right)-\frac{2}{(2\pi k)^{8}}\left(p_{n-7}^{[m-1]}(1^{-}) - p_{n-7}^{[m-1]}(0)\right)\\
& &\vdots \\
& &+ (-1)^{\lfloor \frac{n}{2} \rfloor -1}\frac{2}{(2\pi k)^{2\lfloor \frac{n}{2} \rfloor}}\left(p_{n-\left(\lfloor \frac{n}{2} \rfloor -1\right)}^{[m-1]}(1^{-}) - p_{n-\left(\lfloor \frac{n}{2} \rfloor -1\right)}^{[m-1]}(0)\right),
\end{eqnarray*}
and
\begin{eqnarray*}
b_{k,n}^{[m-1]}&=& -\frac{2}{2\pi k}\left(p_{n}^{[m-1]}(1^{-}) - p_{n}^{[m-1]}(0)\right)+\frac{2}{(2\pi k)^{3}}\left(p_{n-2}^{[m-1]}(1^{-}) - p_{n-2}^{[m-1]}(0)\right)\\
& & -\frac{2}{(2\pi k)^{5}}\left(p_{n-4}^{[m-1]}(1^{-}) - p_{n-4}^{[m-1]}(0)\right)+\frac{2}{(2\pi k)^{7}}\left(p_{n-6}^{[m-1]}(1^{-}) - p_{n-6}^{[m-1]}(0)\right)\\
& &\vdots \\
& &+ (-1)^{\lfloor \frac{n}{2} \rfloor + 1}\frac{2}{(2\pi k)^{2\lfloor \frac{n}{2} \rfloor + 1}}\left(p_{n-2\lfloor \frac{n}{2} \rfloor}^{[m-1]}(1^{-}) - p_{n-2\lfloor \frac{n}{2} \rfloor}^{[m-1]}(0)\right).
\end{eqnarray*}

From these last relations we obtain \eqref{Fouriers3} and \eqref{Fouriers4}, respectively.

\end{proof}

\begin{theorem}
\label{pzeta1}
For a fixed $m\in \NN$ and $n\in\NN$,  let $p_{n}^{[m-1]}(x)$  be the periodic generalized Bernoulli functions of level $m$. If $x\in (0,1)$, then  the following identity holds.
\begin{eqnarray}
\nonumber
p_{n}^{[m-1]}(x)&=&\frac{m!(m-1)}{2(m+1)}\frac{x^{n}}{n!} \\
\nonumber
& & +\sum_{k=1}^{\lfloor \frac{n}{2} \rfloor}{n\choose 2k}\left[(2k)!p_{2k}^{[m-1]}(0)+ \frac{2(-1)^{k}m!(2k)!\zeta(2k)}{(2\pi)^{2k}}\right]\frac{x^{n-2k}}{n!}\\
\label{relazber1}
& & +\sum_{k=1}^{\lfloor \frac{n-1}{2} \rfloor}p_{2k-1}^{[m-1]}(0)\frac{x^{n-2k-1}}{(n-2k-1)!} + m!p_{n}(x).
\end{eqnarray}
\end{theorem}

\begin{proof}
Using Proposition \ref{FourierFPBG} we obtain the following expression for $p_{1}^{[m-1]}(x)$:

\begin{equation}
\label{teo1}
p_{1}^{[m-1]}(x)=\frac{m!(m-1)}{2(m+1)}-m!\sum_{k=1}^{\infty} \frac{2\sin 2\pi kx}{2\pi k}, \quad \mbox{ whenever } \quad x\in(0,1).
\end{equation}

Then in view of \eqref{peri1} and \eqref{[Sec2]-Intf3}, we see that
\begin{equation}
\label{teo2}
p_{n}^{[m-1]}(x)= p_{n}^{[m-1]}(0) + \int_{0}^{x} p_{n-1}^{[m-1]}(t)\,dt,\quad \mbox{ if } \quad x\in[0,1).
\end{equation}

Taking $n=2$ and substituting  \eqref{teo1} into \eqref{teo2}, we get
\begin{eqnarray*}
\begin{split}
p_{2}^{[m-1]}(x)&= p_{2}^{[m-1]}(0) + \int_{0}^{x} p_{1}^{[m-1]}(t)\,dt \\
&= \frac{m!(m-1)}{2(m+1)}x +\left(2!p_{2}^{[m-1]}(0)-\frac{2m!2!}{(2\pi)^{2}}\zeta(2)\right)\frac{1}{2!} + m!p_{2}(x).
\end{split}
\end{eqnarray*}

Similarly, for $n=3$ we can deduce
\begin{eqnarray*}
p_{3}^{[m-1]}(x)&=& p_{3}^{[m-1]}(0) + \int_{0}^{x} p_{2}^{[m-1]}(t)\,dt \\
&=& \frac{m!(m-1)}{2(m+1)}\frac{x^{2}}{2!} + \left(2!p_{2}^{[m-1]}(0)-\frac{2m!2!}{(2\pi)^{2}}\zeta(2)\right)\frac{1}{2!}x + p_{3}^{[m-1]}(0)\\
& &+\,m!p_{3}(x).
\end{eqnarray*}

Iterating this procedure \eqref{relazber1}  follows.

\end{proof}

Recall that the Dirichlet convergence theorem \cite{Foll,Ph,St}
guarantees that the Fourier series \eqref{Fourier1} converges pointwise at $x\in\ZZ$ to the  average of $p_{n}^{[m-1]}(x^{+})$ and $p_{n}^{[m-1]}(x^{-})$. Indeed, based on this fact we prove the next result.

\begin{theorem}
\label{htyq1}
For a fixed $m\in\NN$ and any $r\in\NN$, the following identity holds.
\begin{eqnarray}
\label{eulergen}
\zeta(2r)=\frac{(-1)^{r-1}2^{2r-1} \pi ^{2r}B_{2r}^{[m-1]}}{m!(2r)!}+\Delta_{r}^{[m-1]},
\end{eqnarray}
where
\begin{eqnarray}
\nonumber
\Delta_{r}^{[m-1]}&=&\frac{(-1)^{r-1}2^{2r-1} \pi ^{2r}}{m!}\left[\frac{B_{2r}^{[m-1]}(1) - B_{2r}^{[m-1]}}{2(2r)!}-  \frac{B_{2r+1}^{[m-1]}(1)- B_{2r+1}^{[m-1]}}{(2r+1)!}\right. \\
\label{eulergenn1}
& &-\left. \sum_{j=1}^{r-1}\frac{\left(B_{2r-2j+1}^{[m-1]}(1)- B_{2r-2j+1}^{[m-1]}\right)}{(2r-2j+1)!}\frac{B_{2j}}{(2j)!}\right]\\
\nonumber
&=&\frac{(-1)^{r-1}2^{2r-1} \pi ^{2r}}{m!}\left[\frac{1}{2(2r)!}\sum_{k=0}^{2r-1}\binom{2r}{k} B_{k}^{[m-1]}-\frac{1}{(2r+1)!}\sum_{k=0}^{2r}
\binom{2r+1}{k} B_{k}^{[m-1]}\right.\\
\label{eulergen1}
& &-\left.\sum_{j=1}^{r-1}\sum_{k=0}^{2j}\binom{2j+1}{k}
\frac{B_{k}^{[m-1]}B_{2j}}{(2j+1)!(2j)!}\right].
\end{eqnarray}
\end{theorem}

\begin{proof}
Let us consider $n=2r$ and $x=0$ in \eqref{Fourier1}. Since $x=0$ is a point of discontinuity of  $p_{2r}^{[m-1]}(x)$, by the Dirichlet convergence theorem \cite{Foll,Ph,St} we have

\begin{eqnarray}
\label{auxiliars}
\frac{p_{2r}^{[m-1]}(0^{+}) + p_{2r}^{[m-1]}(0^{-})}{2} = \frac{a_{0,2r}^{[m-1]}}{2} + \sum_{k=1}^{\infty} a_{k,2r}^{[m-1]}.
\end{eqnarray}

Since
$$\frac{p_{2r}^{[m-1]}(0^{+}) + p_{2r}^{[m-1]}(0^{-})}{2} = \frac{B_{2r}^{[m-1]} + B_{2r}^{[m-1]}(1)}{2(2r)!},$$
using  \eqref{Fourier2} and \eqref{Fouriers3}, we can rewrite \eqref{auxiliars} as follows

\begin{eqnarray}
\label{auxiliars2}
\nonumber
\frac{B_{2r}^{[m-1]} + B_{2r}^{[m-1]}(1)}{2(2r)!} &=& \frac{1}{(2r+1)!}\left(B_{2r+1}^{[m-1]}(1)- B_{2r+1}^{[m-1]}\right)\\
& &+ \sum_{k=1}^{\infty} \sum_{j=0}^{r-1} (-1)^{j}\frac{2}{(2\pi k)^{2j+2}}\frac{\left(B_{2r-2j-1}^{[m-1]}(1)- B_{2r-2j-1}^{[m-1]}\right)}{(2r-2j-1)!}.
\end{eqnarray}

Taking into account that $p_{2j+2}(0) = (-1)^{j+1}\sum_{n=1}^{\infty} \frac{2}{(2\pi n)^{2j+2}}$, the relation \eqref{auxiliars2} can be expressed as

\begin{eqnarray*}
\frac{B_{2r}^{[m-1]} + B_{2r}^{[m-1]}(1)}{2(2r)!} &=& \frac{1}{(2r+1)!}\left(B_{2r+1}^{[m-1]}(1)- B_{2r+1}^{[m-1]}\right)\\
& &+\sum_{j=0}^{r-2}\frac{\left(B_{2r-2j-1}^{[m-1]}(1)- B_{2r-2j-1}^{[m-1]}\right)}{(2r-2j-1)!}\frac{B_{2j+2}}{(2j+2)!}+m!p_{2r}(0).
\end{eqnarray*}

Or equivalently,
\begin{eqnarray}
\label{auxiliars4}
\nonumber
\frac{B_{2r}^{[m-1]} + B_{2r}^{[m-1]}(1)}{2(2r)!} &=& \frac{1}{(2r+1)!}\left(B_{2r+1}^{[m-1]}(1)- B_{2r+1}^{[m-1]}\right)\\
& &+\sum_{j=1}^{r-1}\frac{\left(B_{2r-2j+1}^{[m-1]}(1)- B_{2r-2j+1}^{[m-1]}\right)}{(2r-2j+1)!}\frac{B_{2j}}{(2j)!}+m!p_{2r}(0).
\end{eqnarray}

Now, from \eqref{auxiliars4} we deduce that
\begin{eqnarray}
\label{auxiliars5}
\nonumber
\frac{2(-1)^{r-1}\zeta(2r)}{(2 \pi)^{2r}}&=&\frac{1}{m!}\left[\frac{B_{2r}^{[m-1]}(1) + B_{2r}^{[m-1]}}{2(2r)!}-  \frac{B_{2r+1}^{[m-1]}(1)- B_{2r+1}^{[m-1]}}{(2r+1)!}\right. \\
& &-\left. \sum_{j=1}^{r-1}\frac{\left(B_{2r-2j+1}^{[m-1]}(1)- B_{2r-2j+1}^{[m-1]}\right)}{(2r-2j+1)!}\frac{B_{2j}}{(2j)!} \right].
\end{eqnarray}

Hence, \eqref{auxiliars5} takes the form:
\begin{eqnarray}
\label{auxiliars6}
\zeta(2r)&=&\frac{(-1)^{r-1}2^{2r-1} \pi ^{2r}B_{2r}^{[m-1]}}{m!(2r)!}+\Delta_{r}^{[m-1]},
\end{eqnarray}
where
\begin{eqnarray*}
\Delta_{r}^{[m-1]}&=&\frac{(-1)^{r-1}2^{2r-1} \pi ^{2r}}{m!}\left[
\frac{B_{2r}^{[m-1]}(1) + B_{2r}^{[m-1]}}{2(2r)!}-  \frac{B_{2r+1}^{[m-1]}(1)- B_{2r+1}^{[m-1]}}{(2r+1)!}\right. \\
& &-\left. \sum_{j=1}^{r-1}\frac{\left(B_{2r-2j+1}^{[m-1]}(1)- B_{2r-2j+1}^{[m-1]}\right)}{(2r-2j+1)!}\frac{B_{2j}}{(2j)!} \right].
\end{eqnarray*}
Hence,  $\Delta_{r}^{[m-1]}$ satisfies \eqref{eulergenn1}.

Finally, the substitution of \eqref{[Sec2]-bsf1} into
the above expression for  $\Delta_{r}^{[m-1]}$, and some suitable computations yield the  identity \eqref{eulergen1}.

\end{proof}

Notice that if  $m=1$ in \eqref{eulergen} then we recover \eqref{equa1}. It is not difficult to see that for $r=1$ the identity \eqref{eulergen} yields the same result than the identity  \eqref{peri12}.

\section{Riemann zeta function and quadrature formulae of Euler-Maclaurin type}
\label{[Section-4]-Basic}

It is well known that using the Euler-Maclaurin summation formula (cf. \cite{A,DR,L,W}, and \cite[Chap. 2, Sec. 3, p. 30]{N}) it is possible to deduce the following formula for the integral of the product of two classical Bernoulli polynomials

\begin{equation}
\label{euler1}
\int_{0}^{1}B_{s}(t)B_{r}(t)dt= (-1)^{s+1}\frac{s!r!}{(s+r)!}B_{s+r}, \quad r,s\geq 1.
\end{equation}

Using integration by parts a similar formula to \eqref{euler1} has been deduced in \cite{QRU}. More precisely, for an integer $r\geq 0$ and a closed interval $[a,b]$, let $C^{r}[a,b]$ denote the set of all $r$-times continuously differentiable functions defined on $[a,b]$. Then following result holds.

\begin{lemma}
\label{lem1} \cite[Lemma 1]{QRU} Let $r\geq 1$ and  $f\in C^{r}[0,1]$. For a fixed $m\in\NN$, we have
\begin{equation}
\label{euler2}
\int_{0}^{1}f(t)dt=\frac{1}{m!}\left[ \sum_{k=1}^{r} A_{k}^{[m-1]}(f)+ \frac{(-1)^{r}}{r!}\int_{0}^{1} f^{(r)}(t)B_{r}^{[m-1]}(t)dt\right],
\end{equation}
where
$$A_{k}^{[m-1]}(f)= \frac{(-1)^{k}}{k!}\left(f^{(k-1)}(0)B_{k}^{[m-1]}- f^{(k-1)}(1)B_{k}^{[m-1]}(1) \right), \quad k=1,\ldots,r.$$
\end{lemma}

Next, applying the substitution $f(t)=B_{r+n}^{[m-1]}(t)$ into  \eqref{euler2} and taking into account \eqref{[Sec2]-AppDer1},    \eqref{[Sec2]-Intf3} we have

\begin{equation}
\label{euler3}
\int_{0}^{1}B_{r}^{[m-1]}(t)B_{n}^{[m-1]}(t)\,dt= \frac{(-1)^{r+1}r!n!m!}{(r+n)!}\left[ \frac{B_{r+n+1}^{[m-1]}-B_{r+n+1}^{[m-1]}(1)}{r+n+1}+ \frac{1}{m!}\sum_{k=1}^{r}A_{k}^{[m-1]}\right], \quad r,n\geq 1,
\end{equation}
where
$$A_{k}^{[m-1]}= \frac{(-1)^{k}}{k}\binom{r+n}{k-1} \left(B_{r+n-k+1}^{[m-1]}B_{k}^{[m-1]}- B_{r+n-k+1}^{[m-1]}(1)B_{k}^{[m-1]}(1)\right), \quad k=1,\ldots,r.$$

The expression \eqref{euler3} is the analogue of \eqref{euler1} in the setting of the generalized Bernoulli polynomials of level $m$.
We strongly recommend to the interested reader see \cite{QRU} for the corresponding proofs of the results mentioned above.

\medskip

Let $L^{2}[0,1]$ be the space of the square-integrable functions on $[0,1]$, endowed with the norm
$$\|f\|_{L^{2}[0,1]}:=\left(\int_{0}^{1}|f(t)|^{2}dt\right)^{1/2}= \langle f,f\rangle^{1/2},$$
where
$$\langle f, g \rangle:=\int_{0}^{1}f(t)g(t)dt,  \mbox{ for every }  f,g\in L^{2}[0,1].$$

It is  not difficult  to see  that  we  can determine the norm $\|B_{n}^{[m-1]}\|_{L^{2}[0,1]}$ using \eqref{euler3}, as

\begin{eqnarray}
\nonumber
\|B_{n}^{[m-1]}\|_{L^{2}[0,1]}^{2}&=& \frac{(n!)^{2}m!(-1)^{n}}{(2n+1)!}(B_{2n+1}^{[m-1]}(1)-B_{2n+1}^{[m-1]})\\
\label{parseval1}
& & + (n!)^{2}(-1)^{n+1}\sum_{k=1}^{n}\frac{(-1)^{k}}{(2n+1-k)!k!}
(B_{2n+1-k}^{[m-1]}B_{k}^{[m-1]}-B_{2n+1-k}^{[m-1]}(1)B_{k}^{[m-1]}(1)).
\end{eqnarray}

From the trigonometric form of Fourier expansion for $f\in L^{2}[0,1]$ it is possible to deduce the following form of Parseval's identity:
\begin{equation}
\label{parseval2}
\|f\|_{L^{2}[0,1]}^{2}= \frac{\left|a_{0}(f)\right|^{2}}{4}+\frac{1}{2}\sum_{k=1}^{\infty}\left|a_{k}(f)\right|^{2}+\left|b_{k}(f)\right|^{2},
\end{equation}
where
\begin{eqnarray*}
a_{k}(f) &=& 2\int_{0}^{1}f(x)\cos(2\pi kx)\,dx,\quad k\geq 0, \\
b_{k}(f) &=& 2\int_{0}^{1}f(x)\sin(2\pi kx)\,dx,\quad k\geq 1.
\end{eqnarray*}

Hence, using  \eqref{parseval1} we show how linear combinations of the values of $\zeta(2k)$ can be obtained by applying Parseval's identity \eqref{parseval2} with the Fourier coefficients  \eqref{Fourier2}, \eqref{Fouriers3} and \eqref{Fouriers4} of the periodic generalized Bernoulli functions of level $m$.

\medskip

Applying Parseval's identity \eqref{parseval2} to $p_{n}^{[m-1]}(x)$ and using \eqref{Fourier2}-\eqref{Fouriers4}, we can deduce that

\begin{eqnarray}
\nonumber
\|B_{n}^{[m-1]}\|_{L^{2}[0,1]}^{2}&=& (n!)^{2}\left[\frac{\left(a_{0,n}^{[m-1]}\right)^{2}}{4}+\frac{1}{2}
\sum_{k=1}^{\infty}\left(a_{k,n}^{[m-1]}\right)^{2}+\left(b_{k,n}^{[m-1]}\right)^{2} \right]\\
\nonumber
&=& \frac{(B_{n+1}^{[m-1]}(1)-B_{n+1}^{[m-1]})^{2}}{(n+1)^{2}}+ 2(n!)^{2}\sum_{k=1}^{\infty}\left[\sum_{j=0}^{\lfloor\frac{n}{2}\rfloor-1}
\frac{(-1)^{j}}{(2\pi k)^{2j+2}}\left(\frac{B_{n-2j-1}^{[m-1]}(1)-B_{n-2j-1}^{[m-1]}}{(n-2j-1)!}\right)\right]^{2}\\
\label{parseval3}
& & + 2(n!)^{2}\sum_{k=1}^{\infty}\left[\sum_{j=0}^{\lfloor\frac{n}{2}\rfloor}
\frac{(-1)^{j+1}}{(2\pi k)^{2j+1}}\left(\frac{B_{n-2j}^{[m-1]}(1)-B_{n-2j}^{[m-1]}}{(n-2j)!}\right)\right]^{2}.
\end{eqnarray}

Comparing \eqref{parseval1} with \eqref{parseval3} we  obtain  the  next equality:

\begin{eqnarray}
\nonumber
\sum_{k=1}^{\infty}A_{k,n}^{2}+B_{k,n}^{2} &=&\frac{m!(-1)^{n}}{2(2n+1)!}(B_{2n+1}^{[m-1]}(1)-B_{2n+1}^{[m-1]}) -\frac{(B_{n+1}^{[m-1]}(1)-B_{n+1}^{[m-1]})^{2}}{2(n+1)^{2}}\\
\label{parseval4}
& & +\frac{(-1)^{n+1}}{2}\sum_{k=1}^{n}\frac{(-1)^{k}}{(2n+1-k)!k!}
(B_{2n+1-k}^{[m-1]}B_{k}^{[m-1]}-B_{2n+1-k}^{[m-1]}(1)B_{k}^{[m-1]}(1)),
\end{eqnarray}
where
\begin{eqnarray*}
A_{k,n} &=& \sum_{j=0}^{\lfloor\frac{n}{2}\rfloor-1}
\frac{(-1)^{j}}{(2\pi k)^{2j+2}}\left(\frac{B_{n-2j-1}^{[m-1]}(1)-B_{n-2j-1}^{[m-1]}}{(n-2j-1)!}\right), \\
B_{k,n}&=& \sum_{j=0}^{\lfloor\frac{n}{2}\rfloor}
\frac{(-1)^{j+1}}{(2\pi k)^{2j+1}}\left(\frac{B_{n-2j}^{[m-1]}(1)-B_{n-2j}^{[m-1]}}{(n-2j)!}\right).
\end{eqnarray*}

Furthermore, if  $m=1$ in  \eqref{parseval4} we recover \eqref{equa1}.

\medskip

Following the ideas of \cite{QU} we can obtain a quadrature formulae of Euler-Maclaurin type based on generalized Bernoulli polynomials of level $m\in \NN\setminus\{1\}$.

\begin{theorem}
\label{htyq2}
Let $r\geq 1$, $f\in C^{r}[a,b]$ and $m\in \NN$. For a fixed $n\in \NN$ let $x_{j}=a+jh$, $j= 0, 1,\ldots,n$, where $h =\frac{b-a}{n}$, and  $f_{j}^{(k-1)}=f^{(k-1)}(x_{j}),$ $k=1,2,\ldots,r$. Then, the following composite trapezoidal rules hold.
\begin{equation}
\label{qgenber3}
\int_{a}^{b}f(t)\,dt  = \sum_{j=0}^{n-1}\sum_{k=1}^{r}\tilde{A}_{k,j}^{[m-1]} (f)+R_{r}^{[m-1]}(f),
\end{equation}
where
$$\tilde{A}_{k,j}^{[m-1]} (f) = \frac{(-1)^{k+1}}{m!k!}h^{k}\left(f_{j+1}^{(k-1)}B_{k}^{[m-1]}(1)-f_{j}^{(k-1)}B_{k}^{[m-1]}\right), \quad 1\leq k \leq r,$$
and
$$R_{r}^{[m-1]}(f)= \frac{(-h)^{r}}{m!r!}\int_{a}^{b}f^{(r)}(t)B_{r}^{[m-1]}\left(\frac{t-a}{h}-\left\lfloor \frac{t-a}{h}\right\rfloor\right)\,dt.$$
\end{theorem}

\begin{proof}
Let $g\in C^{r}[0,1]$. By \eqref{euler2} we get

\begin{eqnarray}
\nonumber
\label{aux10}
\int_{0}^{1}g(t)\,dt  & = &\frac{1}{m!}\sum_{k=1}^{r}\frac{(-1)^{k+1}}{k!}\left(g^{(k-1)}(1)B_{k}^{[m-1]}(1)-g^{(k-1)}(0)B_{k}^{[m-1]}\right)\\
& & +\frac{(-1)^{r}}{m!r!}\int_{0}^{1}g^{(r)}(t)B_{r}^{[m-1]}(t)\,dt.
\end{eqnarray}

Taking $g(t)= f(x_{j}+ht)$ it is easy to check that $g^{(k)}(t)=h^{k}f^{(k)}(x_{j}+ht)$ for $k= 1,2,\ldots,r$. Substituting $g^{(k-1)}(1)$, $g^{(k-1)}(0)$, $g^{(r)}(t)$ into \eqref{aux10}, and making a suitable change of variable, we obtain that
\begin{eqnarray}
\nonumber
\int_{x_{j}}^{x_{j+1}}f(t)\,dt  & = &
\frac{1}{m!}\sum_{k=1}^{r}\frac{(-1)^{k+1}}{k!}h^{k}\left(f^{(k-1)}(x_{j+1})B_{k}^{[m-1]}(1)-f^{(k-1)}(x_{j})B_{k}^{[m-1]}\right)\\
\label{aux11}
& & +\frac{(-h)^{r}}{m!r!}\int_{x_{j}}^{x_{j+1}}f^{(r)}(t)B_{r}^{[m-1]}\left(\frac{t-x_{j}}{h}\right)\,dt,
\end{eqnarray}
whenever $j= 0, 1,\ldots,n-1$.

\medskip

Next, adding all these terms for $j= 0, 1,\ldots,n-1$ to both sides of \eqref{aux11}, and nothing that if  $x_{j}\leq t \leq x_{j+1}$ then $j\leq \frac{t-a}{h}\leq j+1$, we have
\begin{eqnarray*}
\int_{a}^{b}f(t)\,dt & = &\sum_{j=0}^{n-1}\int_{x_{j}}^{x_{j+1}}f(t)\,dt \\
& = &\frac{1}{m!}\sum_{j=0}^{n-1}\sum_{k=1}^{r}\frac{(-1)^{k+1}}{k!}h^{k}\left(f^{(k-1)}(x_{j+1})B_{k}^{[m-1]}(1)-f^{(k-1)}(x_{j})B_{k}^{[m-1]}\right)\\
& & + \frac{(-h)^{r}}{m!r!}\int_{a}^{b}f^{(r)}(t)B_{r}^{[m-1]}\left(\frac{t-a}{h}-\left\lfloor \frac{t-a}{h}\right\rfloor\right)\,dt.
\end{eqnarray*}

From this last equation \eqref{qgenber3} follows.

\end{proof}

\medskip

We conclude this section with a result that reveals an interesting property about the applications of the quadrature formulae of Euler-Maclaurin type \eqref{qgenber3}. Using the approach given in
\cite[pp. 117-120]{L}, it is possible to provide a theorem
comparing simultaneously the convergence of a series $\sum_{k=1}^{\infty}f(k)$ and an integral $\int_{1}^{\infty}f(x)dx$ in the setting of generalized Bernoulli polynomials of level $m$. In particular, with such a theorem we can estimate the values $\zeta(2k+1)$, for $k\geq 1$.

\medskip

Let $r\geq 1$, $f\in C^{r}[1,\infty)$. For a fixed $m\in\NN$, we will denote by

\begin{eqnarray}
\label{notaci1}
S(l)& :=& \sum_{j=1}^{l}f(j),\\
\label{notaci2}
\widetilde{\sigma}_{r}^{[m-1]}(q_{1})&:=& f(q_{1})+ \frac{1}{m!}\sum_{k=1}^{r}\frac{(-1)^{k+1}}{k!}f^{(k-1)}(q_{1})B_{k}^{[m-1]},\\
\label{notaci3}
\sigma_{r}^{[m-1]}(q_{2})& :=&\frac{1}{m!}\sum_{k=1}^{r}\frac{(-1)^{k+1}}{k!}f^{(k-1)}(q_{2})B_{k}^{[m-1]}(1),\\
\label{notaci4}
\rho_{r}^{[m-1]}(q_{1},q_{2})& :=& \frac{1}{m!}\sum_{j= q_{1}+1}^{q_{2}-1}\sum_{k=2}^{r} \frac{(-1)^{k+1}}{k!}(B_{k}^{[m-1]}(1)-B_{k}^{[m-1]})f^{(k-1)}(j),\\
\label{notaci5}
\quad R_{r}^{[m-1]}(q_{1},q_{2})&:=& \frac{(-1)^{r}}{r!}\int_{q_{1}}^{q_{2}}f^{(r)}(t)B_{r}^{[m-1]}(t-\lfloor t \rfloor)\,dt,
\end{eqnarray}
where $l,q_{1},q_{2} \in\NN$. As well as, we will consider the following limits:

\begin{eqnarray*}
S(\infty)&:=&\lim_{l\to\infty}S(l),\\
\sigma_{r}^{[m-1]}(\infty)&:=&\lim_{q_{2}\to\infty}\sigma_{r}^{[m-1]}(q_{2}),\\
\rho_{r}^{[m-1]}(q_{1},\infty)&:=&\lim_{q_{2}\to\infty}\rho_{r}^{[m-1]}(q_{1},q_{2}),\\
R_{r}^{[m-1]}(q_{1},\infty)&:=&\lim_{q_{2}\to\infty}R_{r}^{[m-1]}(q_{1},q_{2}),\\
e_{r}^{[m-1]}(q_{1})&:=&\rho_{r}^{[m-1]}(q_{1},\infty),\\
\delta_{r}^{[m-1]}(q_{1})&:=& R_{r}^{[m-1]}(q_{1},\infty).
\end{eqnarray*}

For the reader's convenience, we  recall  the definition of Euler's constant for a function $f$ (cfr. \cite[p. 118]{L}). For $f\in C^{r}[1,\infty)$ and any $n\in\NN$ let us consider the sequence
\begin{equation}
\label{eulerq}
\gamma_{n}(f):=\sum_{i=1}^{n}f(i)-\int_{1}^{n}f(t)\,dt.
\end{equation}
Euler's constant for function $f$  is defined as the limit
\begin{equation}
\label{eulerq1}
\gamma(f):=\lim_{n\to\infty}\gamma_{n}(f),
\end{equation}
whenever such limit exists and be finite.

\medskip

The quadrature formulae of Euler-Maclaurin type \eqref{qgenber3} is also of theoretical interest. More precisely, the definitions  \eqref{eulerq}, \eqref{eulerq1} and the formulae \eqref{qgenber3} imply the following result:

\begin{theorem}
\label{teoconv}
For a fixed $m\in\NN$, every $r, p, n\in \NN$ and $f\in C^{r}[1,\infty)$. Assume that $\rho_{r}^{[m-1]}(1,\infty)$, $\,\int_{1}^{\infty}|f^{(r)}(t)|\,dt$ converge, and the finite limit $\lambda_{0}:=\lim_{n\to\infty}f(n)\,$ exists, then
\begin{enumerate}
  \item[(a)] The integral $\int_{1}^{\infty}f(t)\,dt$ converges if and only if the series  $\sum_{j=1}^{\infty}f(j)$ converges.
\item[(b)] If the integral $\int_{1}^{\infty}f(t)\,dt$ converges, then
\begin{eqnarray}
\nonumber
\int_{1}^{\infty}f(t)\,dt&=&\int_{1}^{p}f(t)\,dt + S(\infty)-S(p-1)+ \sigma_{r}^{[m-1]}(\infty)-\widetilde{\sigma}_{r}^{[m-1]}(p)+ e_{r}^{[m-1]}(p)\\
& & +\,\delta_{r}^{[m-1]}(p).
\end{eqnarray}
\end{enumerate}
\end{theorem}

Notice that if $\rho_{r}^{[m-1]}(1,\infty)$ converges, then $\lim_{n\to\infty}f^{(k-1)}(n)=0$ for every $k=2,\ldots,r$.

\begin{proof}
Without loss of generality we can assume that $p\leq n$. The substitution $a=p$, $b=n$ and $h=1$ into \eqref{qgenber3} yields the identity
\begin{eqnarray}
\nonumber
\label{cdr1}
\int_{p}^{n}f(t)\,dt &=& \sum_{j=p}^{n-1}f(j) + \frac{1}{m!}\sum_{k=1}^{r}\frac{(-1)^{k+1}}{k!}f^{(k-1)}(n)B_{k}^{[m-1]}(1)\\
\nonumber
& &-\frac{1}{m!}\sum_{k=1}^{r}\frac{(-1)^{k+1}}{k!}f^{(k-1)}(p)B_{k}^{[m-1]}-f(p)\\
\nonumber
& & + \frac{1}{m!}\sum_{j=p+1}^{n-1}\sum_{k=2}^{r} \frac{(-1)^{k+1}}{k!}(B_{k}^{[m-1]}(1)-B_{k}^{[m-1]})f^{(k-1)}(j)\\
& &+ \frac{(-1)^{r}}{r!}\int_{p}^{n}f^{(r)}(t)B_{r}^{[m-1]}(t-\lfloor t \rfloor)\,dt.
\end{eqnarray}

By \eqref{notaci1}-\eqref{notaci5} we can rewrite \eqref{cdr1} as follows
\begin{eqnarray}
\label{cdr2}
\nonumber
\int_{p}^{n}f(t)\,dt &=& S(n-1)-S(p-1)-f(n)+\sigma_{r}^{[m-1]}(n)-\widetilde{\sigma}_{r}^{[m-1]}(p)\\
& &+ \rho_{r}^{[m-1]}(p,n)+ R_{r}^{[m-1]}(p,n),
\end{eqnarray}
where $S(0)=0$ by definition.

\medskip

The remainder $R_{r}^{[m-1]}(p,n)$ can be estimated by
\begin{equation}
\label{cdr3}
\left|R_{r}^{[m-1]}(p,n)\right|\leq \frac{\mu_{r}^{[m-1]}}{r!}\int_{p}^{n}|f^{(r)}(t)|\,dt,
\end{equation}
where $\mu_{r}^{[m-1]}=\max\{|B_{r}^{[m-1]}(x)|:0\leq x \leq 1\}.$

\medskip

By \eqref{eulerq} and the formula \eqref{cdr2} we obtain
\begin{equation}
\label{cdr4}
\gamma_{n}(f)= f(n)+\widetilde{\sigma}_{r}^{[m-1]}(1)-\sigma_{r}^{[m-1]}(n)-\rho_{r}^{[m-1]}(1,n)- R_{r}^{[m-1]}(1,n).
\end{equation}

Our assumptions imply, according to \eqref{cdr4}, that  the Euler's constant for the function $f$, $\gamma(f)$, exists and  the next equality is satisfied:

\begin{equation}
\label{cdr5}
\gamma(f) = \lambda_{0}+\widetilde{\sigma}_{r}^{[m-1]}(1)-\sigma_{r}^{[m-1]}(\infty)-\rho_{r}^{[m-1]}(1,\infty)- R_{r}^{[m-1]}(1,\infty).
\end{equation}

Now, from \eqref{cdr4} and \eqref{cdr5} we have
\begin{equation}
\label{cdr6}
\gamma(f)= \gamma_{n}(f)+\lambda_{0}-f(n)+\sigma_{r}^{[m-1]}(n)-\sigma_{r}^{[m-1]}(\infty)-e_{r}^{[m-1]}(n-1)- \delta_{r}^{[m-1]}(n),
\end{equation}
where
\begin{equation*}
\left|\delta_{r}^{[m-1]}(n)\right|\leq \frac{\mu_{r}^{[m-1]}}{r!}\int_{n}^{\infty}\left|f^{(r)}(t)\right|\,dt.
\end{equation*}

Thus, substituting \eqref{eulerq} into \eqref{cdr6} and using \eqref{notaci1} we obtain

\begin{eqnarray}
\label{cdr7}
\nonumber
\int_{1}^{n}f(t)\,dt &=& S(n)-\gamma(f) +\lambda_{0}-f(n) + \sigma_{r}^{[m-1]}(n)-\sigma_{r}^{[m-1]}(\infty)-e_{r}^{[m-1]}(n-1)\\
& &- \delta_{r}^{[m-1]}(n).
\end{eqnarray}

Finally, part $(a)$ of Theorem \ref{teoconv} can be deduced from \eqref{cdr7}. In order to obtain  part $(b)$ of Theorem \ref{teoconv} it suffices to consider \eqref{cdr2} and the equality $\int_{p}^{n}f(t)\,dt=\int_{1}^{n}f(t)\,dt-\int_{1}^{p}f(t)\,dt$.

\end{proof}

The interested reader may consult the analogous result for $m=1$ in \cite[Theorem 2]{L}.

\medskip

\begin{example}
\label{examp1} To compute $\zeta(3)=\sum_{k=1}^{\infty}\frac{1}{k^{3}}$, we can put $m=5$, $r=2$, $p=100$, $f(x)=\frac{1}{x^{3}}$, $x\in [1,\infty)$ and apply part (b) of Theorem \ref{teoconv}.  Then, we obtain
\begin{eqnarray*}
S(99)&=& 1.2020064006596776104,\\
\sigma_{2}^{[4]}(p)&=& \frac{5}{6p^{3}}+ \frac{85}{84p^{4}},\\
\sigma_{2}^{[4]}(100)&=& 8.4345238095238095238\times10^{-7},\\
\sigma_{2}^{[4]}(\infty) &=& 0,\\
\widetilde{\sigma}_{2}^{[4]}(p)&=& \frac{5}{6p^{3}}+\frac{1}{84p^{4}},\\
\widetilde{\sigma}_{2}^{[4]}(100)&=& 8.3345238095238095238\times10^{-7},\\
e_{2}^{[4]}(100)&=&3.2836666500022217224\times10^{-7}.
\end{eqnarray*}
Next, part (b) of Theorem \ref{teoconv} gives
\begin{eqnarray*}
\zeta(3) &=& \int_{100}^{\infty}\frac{dt}{t^{3}}+S(99)- \sigma_{2}^{[4]}(\infty)+\widetilde{\sigma}_{2}^{[4]}(100)-e_{2}^{[4]}(100)-\delta_{2}^{[4]}(100)\\
&=& 0.00005+ 1.2020064006596776104\\
& & +(8.3345238095238095238-3.2836666500022217224)\times10^{-7}
-\delta_{2}^{[4]}(100)\\
&=&1.2020560622930126102-\delta_{2}^{[4]}(100).
\end{eqnarray*}

Since
$$ \delta_{2}^{[4]}(100)\approx3.10296\times 10^{-7},$$
we obtain the following estimates for $\zeta(3)$:
\begin{equation}
\label{estima1}
\zeta(3) \approx 1.2020557519970993510.
\end{equation}
In this case, our approximation is accurate up to five decimal places of $\zeta(3)= 1.2020569031595942854...$.

Since, for $p\geq 1$,
$$\left|\delta_{2}^{[4]}(p)\right|\leq \frac{\mu_{2}^{[4]}}{2}\int_{p}^{\infty}\frac{12}{t^{5}}\,dt=\frac{850}{7p^{4}}.$$
Then,
$$\left|\delta_{2}^{[4]}(100)\right|\leq 0.000001214285714,$$
and the estimate \eqref{estima1} could be refined in order to get an accurate up to six decimal places.
\end{example}

\begin{example}
\label{examp2} Now, we will estimate $\zeta(3)=\sum_{k=1}^{\infty}\frac{1}{k^{3}}$, taking $m=2$, $r=2$, $p=20$, $f(x)=\frac{1}{x^{3}}$, $x\in [1,\infty)$ and  apply part (b) of Theorem \ref{teoconv} again.  In this case, we have
\begin{eqnarray*}
S(19)&=& 1.2020064006596776104,\\
\sigma_{2}^{[1]}(p)&=& \frac{2}{3p^{3}}+ \frac{7}{12p^{4}},\\
\sigma_{2}^{[1]}(20)&=& 0.000086979166666666666667,\\
\sigma_{2}^{[1]}(\infty) &=& 0,\\
\widetilde{\sigma}_{2}^{[1]}(p)&=&\frac{2}{3p^{3}}+\frac{1}{12p^{4}},\\
\widetilde{\sigma}_{2}^{[1]}(20)&=&0.000083854166666666666667,\\
e_{2}^{[1]}(20)&=& 0.00002244785177830327.
\end{eqnarray*}
From  part (b) of Theorem  \ref{teoconv} we get
\begin{eqnarray*}
\zeta(3) &=& \int_{20}^{\infty}\frac{dt}{t^{3}}+S(19)- \sigma_{2}^{[1]}(\infty)+\widetilde{\sigma}_{2}^{[1]}(20)-e_{2}^{[1]}(20)-\delta_{2}^{[1]}(20)\\
&=& 0.00125+ 1.2007428419584369581\\
& & +0.000083854166666666666667-0.00002244785177830327
-\delta_{2}^{[1]}(20)\\
&=&1.2020560522930126102-\delta_{2}^{[1]}(20).
\end{eqnarray*}

Since
$$ \delta_{2}^{[1]}(20)\approx9.40\times 10^{-7},$$
we obtain the following numerical approximation of $\zeta(3)$
\begin{equation}
\label{estima2}
\zeta(3) \approx 1.2019663288791965826,
\end{equation}
which only is accurate up to two decimal places of
$\zeta(3)= 1.2020569031595942854...$.
\end{example}

\begin{example}
\label{examp3}
To estimate  $\zeta(5)=\sum_{k=1}^{\infty}\frac{1}{k^{5}}$, we put $m=2$, $r=6$, $p=30$, $f(x)=\frac{1}{x^{5}}$, $x\in [1,\infty)$ and apply part (b) of Theorem \ref{teoconv}.  In this case, we have
\begin{eqnarray*}
S(29)&=& 1.0369274253541474188,\\
\sigma_{6}^{[1]}(p)&=&
\frac{2}{3p^{5}}+\frac{35}{36p^{6}}+\frac{8}{9p^{7}}+\frac{77}{216p^{8}}-\frac{26}{81p^{9}}-\frac{151}{270p^{10}},\\
\sigma_{6}^{[1]}(30)&=& 2.8809650704405569010\times10^{-8},\\
\sigma_{6}^{[1]}(\infty) &=& 0,\\
\widetilde{\sigma}_{6}^{[1]}(p)&=&
\frac{2}{3p^{5}}+\frac{5}{36p^{6}}+\frac{1}{18p^{7}}-\frac{7}{216p^{8}}-\frac{5}{81p^{9}}-\frac{1}{270p^{10}},\\
\widetilde{\sigma}_{6}^{[1]}(30)&=&2.7627849714267435143\times10^{-8},\\
e_{6}^{[1]}(30)&=& 6.48060252152\times10^{-9}.
\end{eqnarray*}
From  part (b) of Theorem  \ref{teoconv} we get
\begin{eqnarray*}
\zeta(5) &=& \int_{30}^{\infty}\frac{dt}{t^{5}}+S(29)- \sigma_{3}^{[1]}(\infty)+\widetilde{\sigma}_{3}^{[1]}(30)-e_{3}^{[1]}(30)-\delta_{3}^{[1]}(30)\\
&=& 3.0864197530864197531\times10^{-7}+ 1.0369274253541474188\\
& & +2.7627902250673169740\times10^{-8}-6.47839130112\times10^{-9}
-\delta_{6}^{[1]}(30)\\
&=&1.0369277263337192158-\delta_{6}^{[1]}(30).
\end{eqnarray*}

According to
$$ \delta_{6}^{[1]}(30)\approx
-3.9236379933251\times10^{-51},$$
we obtain the following numerical approximation of $\zeta(5)$
\begin{equation}
\label{estima3}
\zeta(5) \approx 1.0369277263337192158.
\end{equation}
So, our approximation is accurate up to seven decimal places of
$\zeta(5)= 1.0369277551433699263...$.
\end{example}

In \cite[Example 5]{L} the examples \ref{examp1} and \ref{examp2} are considered for the level $m=1$. Indeed, putting $r=2$ and $p=20$ the estimate \eqref{estima1} is also obtained. So, from a numerical viewpoint the level $m=1$ seems to provide a  low computational cost.

\medskip

Finally, the numerical evidence corresponding to the examples \ref{examp1}-\ref{examp3} suggests that when $m>1$ for obtaining
higher precision for our approximations to the series  $\sum_{j=1}^{\infty}f(j)$ we need only use higher values of $r$ in part (b) of Theorem  \ref{teoconv}.

\begin{example}
\label{examp4} Using part (a) of Theorem \ref{teoconv}  we can deduce that the series
$$\sum_{k=1}^{\infty}\frac{\cos(\sqrt{k})}{k}$$
converges, since
$$\int_{1}^{\infty}\frac{\cos(\sqrt{t})}{t}\,dt\approx -0.67480784580193626932...$$

The above approximation was performed using MAPLE 15. However, it is not difficult to show that the integral  $\int_{1}^{\infty}\frac{\cos(\sqrt{t})}{t}\,dt$ converges.
Notice that
$$2\int_{1}^{b}\frac{d(\sin(\sqrt{t}))}{\sqrt{t}}=\int_{1}^{b}\frac{\cos(\sqrt{t})}{t}\,dt,$$
and by  the formula for integration by parts of Riemann-Stieltjes, we have:
$$2\left[\int_{1}^{b}\frac{d(\sin(\sqrt{t}))}{\sqrt{t}}+ \int_{1}^{b}\sin(\sqrt{t})d\left(\frac{1}{\sqrt{t}}\right)\right] =2\left(\frac{\sin(\sqrt{b})}{\sqrt{b}}-\sin(1)\right).$$
Consequently,
$$\int_{1}^{b}\frac{\cos(\sqrt{t})}{t}\,dt=2\left(\frac{\sin(\sqrt{b})}{\sqrt{b}}-\sin(1)\right)+\int_{1}^{b}\frac{\sin(\sqrt{t})}{t^{3/2}}\,dt,$$
since $\lim_{b\to\infty}\frac{\sin(\sqrt{b})}{\sqrt{b}}=0$ and the integral $\int_{1}^{\infty}\frac{\sin(\sqrt{t})}{t^{3/2}}\,dt$ converges, then
$$\int_{1}^{\infty}\frac{\cos(\sqrt{t})}{t}\,dt \quad \mbox{ converges.}$$

We can provide another solution by using Dirichlet's test for improper integrals (see for instance, \cite[Example 4]{L} where a similar series is considered.)
\end{example}

\bigskip

%\subsection*{Acknowledgements.}

\end{document}